\documentclass[12pt,reqno]{amsart}
\usepackage[T2A]{fontenc}
\usepackage[english]{babel}
\usepackage{amsaddr}
\usepackage{amsmath}
\usepackage{amssymb}
\usepackage{amsfonts}
\usepackage{epsfig}
\usepackage{srcltx}
\usepackage{subfigure}
\usepackage{float}
\usepackage{xcolor}
\usepackage{mathtools}
\usepackage[a4paper, mag=1000, includefoot, left=2cm, right=2cm, top=2cm, bottom=2cm, headsep=1cm, footskip=1cm]{geometry}
   \usepackage[unicode,colorlinks]{hyperref}
     \hypersetup{colorlinks=true, citecolor=blue, linkcolor=blue}
\newtheorem{Th}{Theorem}
\newtheorem{Lem}{Lemma}

\begin{document}

\thispagestyle{empty}

\title[Resonance in isochronous systems with decaying stochastic perturbations]{Resonance in isochronous systems with decaying oscillatory and stochastic perturbations}

\author[O.A. Sultanov]{Oskar A. Sultanov}

\address{
Institute of Mathematics, Ufa Federal Research Centre, Russian Academy of Sciences, Chernyshevsky street, 112, Ufa 450008 Russia.}
\email{oasultanov@gmail.com}


\maketitle

{\small
\begin{quote}
\noindent{\bf Abstract.} 
The combined influence of oscillatory excitations and multiplicative stochastic perturbations of white noise type on isochronous systems in the plane is investigated. It is assumed that the intensity of perturbations decays with time and the excitation frequency satisfies a resonance condition. The occurrence and stochastic stability of solutions with an asymptotically constant amplitude are discussed. By constructing an averaging transformation, we derive a model truncated deterministic system that describes possible asymptotic regimes for perturbed solutions. The persistence of resonant solutions in the phase locking and the phase drifting modes is justified by constructing suitable Lyapunov functions for the complete stochastic system. In particular, we show that decaying stochastic perturbations can shift the boundary of stability domain for resonant solutions.

 \medskip

\noindent{\bf Keywords: }{Isochronous system, decaying perturbation, resonance, phase locking, phase drifting, averaging, stochastic stability}

\medskip
\noindent{\bf Mathematics Subject Classification: }{34F15, 34C15, 34E10, 34C29, 37H30}

\end{quote}
}
{\small

\section*{Introduction}
Oscillatory systems with a fixed constant natural frequency that does not depend on the amplitude are called isochronous~\cite{FC08}. Such systems and their perturbations appear in the study of a wide range of nonlinear problems~\cite{BM61,SJJ92,PRK01,SS18}. It should be noted that resonance effects in isochronous systems under the influence of small periodic disturbances have been studied in many papers. In particular, the bounded solutions were considered in~\cite{Liu09,BF09}, the emergence of unbounded trajectories and the escape from the period annulus were studied in~\cite{OR19,FF05unb,R20}. In this paper, the presence of a small parameter is not assumed and perturbations with the decaying intensity are considered.

The influence of time-decaying disturbances on autonomous systems has been the subject of many prior studies. For example, bifurcations
in asymptotically autonomous systems have been investigated in~\cite{LRS02,KS05,MR08}, the conditions that guarantee the persistence of the autonomous dynamics were discussed in~\cite{LM56,LDP74}, significant changes in dynamics under the influence of damped disturbances were described in~\cite{HRT94}. We consider a special class of decaying oscillatory perturbations with asymptotically constant frequency and discuss the effect of such perturbations on isochronous systems in the plane. Note that this type of disturbances have been studied in several papers. In particular, the occurrence of unbounded resonant solutions for parametrically perturbed harmonic oscillators was treated in~\cite{BN10,PN13}, bifurcations of the equilibrium and various resonant regimes were discussed in~\cite{OS21DCDS}, nonlinear resonance phenomena and the emergence of attractive solutions with asymptotically constant resonant amplitude values were studied in~\cite{OS24QTDS} for isochronous systems and in~\cite{OS25DCDS} for non-isochronous systems. The main focus of this paper is on the consideration of additional effects of stochastic perturbations in such systems. 

It is well known that even small stochastic perturbations can throw the trajectories outside any bounded domain~\cite[Chapter~9]{FW98}. There are a lot of works treated oscillatory systems subject to white noise perturbations with a small constant intensity. For instance, the emergence of solutions with a constant steady-state energy for an isochronous bilinear system subject to an additive white noise was described in~\cite{MD96}, the approximations of solutions for anharmonic oscillators perturbed by noise were studied in \cite{BST94,BL22}, the stochastic averaging of randomly perturbed oscillatory systems with a finite number of critical points was discussed in \cite{MF03}, the approximations for the top Lyapunov exponent were studied in~\cite{AIN04} for a stochastically perturbed oscillator with a single-well potential and in~\cite{PHB04} for the stochastic Duffing–van der Pol equation. The combined influence of small periodic and stochastic perturbations on oscillatory systems was considered in~\cite{RNI86,HZS00,ZW03,HZ04,CNO11}. Effects of decaying stochastic perturbations on scalar autonomous systems were discussed in~\cite{AGR09,ACR11,KT13}. Bifurcations of the equilibrium in near-Hamiltonian systems subject to decaying multiplicative white noise were discussed in~\cite{OS22IJBC,OS24CPAA}. Decaying solutions in systems with damped periodic perturbations and noise were studied in~\cite{OS23SIAM}, while unboundedly growing solutions in systems with fading chirped-frequency excitations and noise were described in~\cite{OS25CNSNS}. However, the influence of white noise perturbations on resonant solutions in isochronous systems have not been studied previously. This is the subject of the present paper.

In this paper, we study the combined effect of oscillatory and multiplicative stochastic perturbations with intensity decaying in time on isochronous systems in the plane. The long-term behaviour of the perturbed trajectories is investigated and the emergence of stochastically stable states with an asymptotically constant non-zero amplitude is discussed. 

The paper is organized as follows. Section~\ref{sec1} provides the statement of the problem. The formulation of the main results is presented in Section~\ref{sec2}, while the justification is contained in subsequent sections. In particular, Section~\ref{sec3} discusses the construction of the averaging transformation that simplify the perturbed system in the first asymptotic terms at infinity in time. The analysis of the simplified system shows that the resonant solutions can occur in two different regimes: a phase locking and a phase drifting. The detailed analysis of the phase locking mode is contained in Section~\ref{sec4}. The phase drifting mode is discussed in Section~\ref{sec5}. In Section~\ref{sex}, the proposed theory is applied to some examples of asymptotically autonomous systems.

\section{Problem statement}\label{sec1}
Consider a system of It\^{o} stochastic differential equations 
\begin{gather}\label{PS}
d\begin{pmatrix}
\varrho \\ \varphi
\end{pmatrix}=
\left({\bf a}_0+{\bf a}(\varrho,\varphi,S(t),t)\right)dt + \varepsilon {\bf A}(\varrho,\varphi,S(t),t)\,d{\bf w}(t), \quad t\geq \tau_0>0,
\end{gather}
where ${\bf w}(t)\equiv (w_1(t),w_2(t))^T$ is a two dimensional Wiener process on a probability space $(\Omega,\mathcal F,\mathbb P)$, ${\bf A}(\varrho,\varphi,S,t)\equiv \{\alpha_{i,j}(\varrho,\varphi,S,t)\}_{2\times 2}$ is a $2\times 2$ matrix, ${\bf a}(\varrho,\varphi,S,t)\equiv (a_1(\varrho,\varphi,S,t),a_2(\varrho,\varphi,S,t))^T$ is a vector function and ${\bf a}_0\equiv (0,\nu_0)^T$, $\nu_0\in\mathbb R_+$. A positive parameter $\varepsilon>0$ controls the intensity of the noise. The deterministic functions $a_i(\varrho,\varphi,S,t)$ and $\alpha_{i,j}(\varrho,\varphi,S,t)$, defined for all $\varrho\in (0, \mathcal R]$, $(\varphi,S)\in\mathbb R^2$, $t>0$, are infinitely differentiable and $2\pi$-periodic with respect to $\varphi$ and $S$. It is assumed that a smooth function $S(t)$ satisfies the asymptotic estimate $S'(t)\sim s_0$ as $t\to\infty$, where $s_0$ is a positive parameter such that the resonance condition holds: there exist coprime integers $\kappa$ and $\varkappa$ such that
\begin{gather}\label{rc}
\kappa s_0=\varkappa \nu_0.
\end{gather}
Note that ${\bf a}(\varrho,\varphi,S(t),t)$ and ${\bf A}(\varrho,\varphi,S(t),t)$ correspond to perturbations of the autonomous system 
\begin{gather*}
\frac{d\hat\varrho}{dt}=0, \quad \frac{d\hat\varphi}{dt}=\nu_0,
\end{gather*}
describing isochronous oscillations on the plane $(x_1,x_2)=(\hat \varrho \cos\hat \varphi, - \hat\varrho\sin\hat\varphi)$ with a constant amplitude $\hat\varrho(t)\equiv \varrho_0$ and a period $T_0=2\pi/\nu_0$. 
It is assumed that the intensity of perturbations decays with time: for each fixed $\varrho$ and $\varphi$
\begin{gather*}
{\bf a}(\varrho,\varphi,S(t),t)\to 0, \quad {\bf A}(\varrho,\varphi,S(t),t)\to 0, \quad t\to\infty.
\end{gather*}
In this case, the perturbed stochastic system \eqref{PS} is asymptotically autonomous. It should be noted that if stochastic part of perturbations is dropped ($\varepsilon =0$), then the phase locking phenomenon in the perturbed deterministic system \eqref{PS} may occur and the attractive solutions with asymptotically constant amplitude for various initial data can arise~\cite{OS24QTDS}. Similar effects in perturbed non-isochronous systems are usually associated with a nonlinear resonance~\cite[Chapter~2]{BVC79}. In the present paper, we take into account the influence of decaying stochastic perturbations on such resonant dynamics.  

Let us specify the class of decaying perturbations. We assume that the following asymptotic expansions hold: 
\begin{gather}\label{fgas}\begin{split}
{\bf a}(\varrho,\varphi,S,t)\sim & \, \sum_{k=n}^\infty   {\bf a}_k(\varrho,\varphi,S) \mu^k(t), \\ 
{\bf A}(\varrho,\varphi,S,t)\sim &\, \sum_{k=p}^\infty  {\bf A}_k(\varrho,\varphi,S) \mu^k(t), \\
S'(t)\sim &\, s_0+\sum_{k=1}^\infty  s_k \mu^k(t)
\end{split}
\end{gather}
as $t\to\infty$ for all $\varrho\in (0, \mathcal R]$ and $(\psi,S)\in\mathbb R^2$, where $n,p\in\mathbb Z_+=\{1,2,\dots\}$, the coefficients ${\bf a}_k(\varrho,\varphi,S)\equiv (a_{1,k}(\varrho,\varphi,S),a_{2,k}(\varrho,\varphi,S))^T$ and ${\bf A}_k(\varrho,\varphi,S)\equiv \{\alpha_{i,j,k}(\varrho,\varphi,S)\}_{2\times 2}$ are $2\pi$-periodic with respect to $\varphi$ and $S$, $s_k$ are real parameters. These series are assumed to be asymptotic as $t\to\infty$. A smooth positive function $\mu(t)$ is strictly decreasing, $\mu(t)\to 0$ as $t\to\infty$ and satisfies the following:  
\begin{gather}\label{mucond}
\exists\, m\in\mathbb Z_+, \chi_m\leq 0: \quad  
\ell(t):=\frac{d}{dt} \log \mu(t) \to 0, \quad  
\frac{\ell(t)}{\mu^{m}(t)}\to \chi_m, \quad 
\frac{\mu^{m+1}(t)}{\ell(t)}\to 0, 
\end{gather}
as $t\to\infty$. Define $\tilde \ell(t):={\ell(t)}{\mu^{-m}(t)}-\chi_m$. It follows from \eqref{mucond} that
\begin{gather*}
\mu'(t)=\mathcal O(\mu^{m+1}(t)), \quad \ell(t)=\mathcal O(\mu^m(t)), \quad \mu^{m+1}(t)=o(\ell(t)), \quad \tilde \ell(t)=o(1), \quad t\to\infty.
\end{gather*}
Note that $\mu(t)\equiv t^{-\alpha}$ with $\alpha\in (0,1]$ satisfies \eqref{mucond} with $\ell(t)\equiv -\alpha t^{-1}$, $m= \lfloor \alpha^{-1} \rfloor$, $\chi_m=0$ if $\alpha^{-1}\not\in\mathbb Z_+$ and $\chi_m=-\alpha$ if $\alpha^{-1}\in\mathbb Z_+$. Another example is given by $\mu(t)\equiv t^{-\alpha} \log t$ with $\alpha\in (0,1]$. In this case, $\ell(t)\equiv - t^{-1}(\alpha-(\log t)^{-1})$, $m=\lfloor \alpha^{-1} \rfloor$ and $\chi_m=0$. Note that such decaying coefficients arise in many nonlinear and non-autonomous problems (see, for example,~\cite{KF13}).

Consider the example
\begin{gather}\label{Ex0}
dx_1=x_2dt, \quad dx_2=\left(-x_1+t^{-\frac{n}{4}} f(x_1,x_2,S(t))\right) dt+\varepsilon t^{-\frac{p}{4}} g(x_1,S(t))\, dw_1(t),
\end{gather}
with 
\begin{gather*}
f(x_1,x_2,S)\equiv  A_1 x_1\cos S+ ( B_0 +  B_1 \sin S) x_2+ C_0 x_2^3, \\
g(x,S)\equiv x \sin S, \quad S(t)\equiv s_0 t+\frac{4s_1}{3}t^{\frac{3}{4}},
\end{gather*} 
and $A_i,B_i,C_i\in\mathbb R$. It can easily be checked that system \eqref{Ex0} in the variables $\varrho=\sqrt{x_1^2+x_2^2}$, $\varphi=-\arctan(x_2/x_1)$ takes the form \eqref{PS} with $\nu_0=1$, $\mu(t)=t^{-1/4}$, 
\begin{gather}\label{fgex0str}
\begin{split}
a_i(\varrho,\varphi,S,t) \equiv & \, \mu^n(t)a_{i,n}(\varrho,\varphi,S)+\mu^{2p}(t) a^\varepsilon_{i,2p}(\varrho,\varphi,S),\\
\alpha_{i,j}(\varrho,\varphi,S,t) \equiv & \, \mu^p(t)\alpha_{i,j,p}(\varrho,\varphi,S),
\end{split}
\end{gather}
where
\begin{equation}\label{fgex0}
\begin{array}{rclrcl}
a_{1,n}&\equiv& -f(\varrho \cos\varphi,-\varrho\sin\varphi,S)\sin\varphi, 
& a^\varepsilon_{1,2p}&\equiv& \varepsilon^2 (2 \varrho)^{-1} g^2(\varrho \cos\varphi,S) \cos^2\varphi , \\
a_{2,n}&\equiv& -\varrho^{-1}f(\varrho \cos\varphi,-\varrho\sin\varphi,S)\cos\varphi,   &
a^\varepsilon_{2,2p}&\equiv & -  \varepsilon^2 (\sqrt 2\varrho )^{-2} g^2(\varrho \cos\varphi,S)\sin 2\varphi ,\\
\alpha_{1,1,p}&\equiv& -g(\varrho \cos\varphi,S)\sin\varphi,  & \alpha_{1,2,p}&\equiv& 0, \\
\alpha_{2,1,p}&\equiv& -\varrho^{-1}g(\varrho \cos\varphi,S)\cos\varphi, & \alpha_{2,2,p}&\equiv& 0.
\end{array}
\end{equation} 
Note that if $f(x_1,x_2,S)\equiv 0 $ and $\varepsilon=0$, then $\varrho(t)\equiv \varrho_0$ and $\varphi(t)\equiv \omega t+\varphi_0$, where $\varrho_0$ and $\varphi_0$ are arbitrary constants. If there is no parametric excitation $ A_1= B_1=\varepsilon=0$ and $ B_0 < 0$, then system \eqref{Ex0} is not self-excited (see~Fig.~\ref{FigEx0}, a). 
By taking $s_0=2$, we see that the resonance condition \eqref{rc} holds with $\kappa=1$, $\varkappa=2$. If $|A_1|+|B_1|\neq 0$ and $ B_0 < 0$, then along with damped solutions, the solutions with a steady-state amplitude may appear such that $\varrho(t)\approx \rho_0$ with some one fixed value $\rho_0>0$ for different initial data (see~Fig.~\ref{FigEx0}, b, c). Thus, the goal of the paper is to describe the conditions that guarantee the existence and stability of such resonant solutions with an asymptotically constant amplitude in the perturbed stochastic system \eqref{PS} satisfying \eqref{rc}, \eqref{fgas} and \eqref{mucond}.

\begin{figure}
\centering
\subfigure[$B_1=\varepsilon=0$]{
\includegraphics[width=0.3\linewidth]{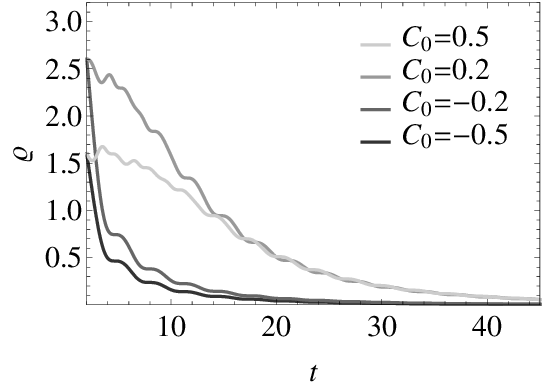}
}
\hspace{1ex}
 \subfigure[$C_0=-0.2$, $\varepsilon=0$]{
 \includegraphics[width=0.3\linewidth]{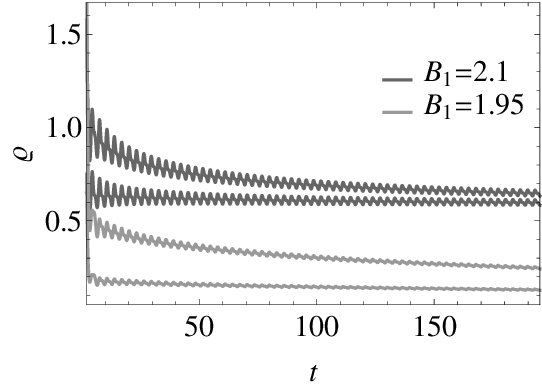}
}
\hspace{1ex}
\subfigure[$C_0=-0.2$, $\varepsilon=2/3$]{
 \includegraphics[width=0.3\linewidth]{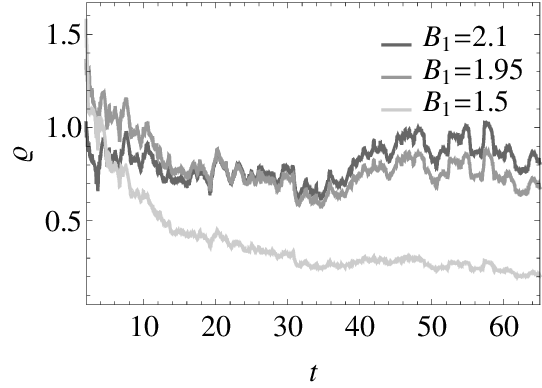}
}
\caption{\small 
The evolution of $\varrho(t)\equiv \sqrt{x_1^2(t)+x_2^2(t)}$ for sample paths of solutions to system \eqref{Ex0} with $s_0=2$, $n=2$, $p=1$, $A_1=0$, $B_0=-1$ and different values of the parameters $B_1$, $C_0$, $\varepsilon$ and initial data.} \label{FigEx0}
\end{figure}

\section{Main results}\label{sec2}

The proposed analysis is based on the simplification of the perturbed equations in the first asymptotic terms as $t\to\infty$. 
Denote by $\langle F(S)\rangle_{\varkappa S}$ the averaging of any function $F(S)$ over $S$ for the interval $[0,2\pi\varkappa]$:
\begin{gather*}
\langle F(S)\rangle_{\varkappa S}\equiv \frac{1}{2\pi\varkappa}\int\limits_0^{2\pi\varkappa} F(S)\,dS.
\end{gather*}
We have the following:
\begin{Th}\label{Th1}
Let system \eqref{PS} satisfy \eqref{rc}, \eqref{fgas} and \eqref{mucond}. Then, for all $N\in [n,m]$ and $\epsilon\in (0,\mathcal R/2)$ there exist $t_0\geq \tau_0$ and the reversible transformations $(\varrho,\varphi)\mapsto (R,\Psi)\mapsto (r,\psi)$,
\begin{align}
\label{ch1} 	R(t)=&\varrho(t), \quad 	& \Psi(t) =& \varphi(t)-\frac{\kappa}{\varkappa}S(t),\\
\label{ch2} r(t)=&R(t)+\tilde U_N(R(t),\Psi(t),t), \quad 	& \psi(t)=& \Psi(t)+\tilde V_N(R(t),\Psi(t),t),
\end{align}
where 
\begin{gather}\label{tildeUNVN}
|\tilde U_N(R,\Psi,t)|\leq \epsilon, \quad |\tilde V_N(R,\Psi,t)|\leq \epsilon, \quad \forall\, R\in(0,\mathcal R], \quad \Psi\in\mathbb R, \quad t\geq t_0,
\end{gather}
and $\tilde U_N(R,\Psi,t)=\mathcal O(\mu(t))$, $\tilde V_N(R,\Psi,t)=\mathcal O(\mu(t))$ as $t\to\infty$ such that system \eqref{PS} can be transformed into
\begin{gather}\label{rpsi}
d\begin{pmatrix}
r \\ \psi
\end{pmatrix} = \begin{pmatrix} \Lambda(r,\psi,S(t),t) \\ 
\Omega(r,\psi,S(t),t)\end{pmatrix}dt+\varepsilon {\bf C}(r,\psi,S(t),t)\, d{\bf w}(t),
\end{gather}
with  $\Lambda(r,\psi,S,t)\equiv\widehat\Lambda_N(r,\psi,t)+ \widetilde\Lambda_N(r,\psi,S,t)$, $\Omega(r,\psi,S,t)\equiv\widehat\Omega_N(r,\psi,t)+\widetilde\Omega_N(r,\psi,S,t)$, ${\bf C}(r,\psi,S,t)\equiv \{\sigma_{i,j}(r,\psi,S,t)\}_{2\times 2}$, defined for all $r\in (0, \mathcal R]$, $(\psi,S)\in\mathbb R^2$, $t\geq t_0$ such that
\begin{gather}
\nonumber
\widehat\Lambda_N(r,\psi,t)\equiv \sum_{k=n}^N \Lambda_k(r,\psi) \mu^k(t), \quad 
\widehat\Omega_N(r,\psi,t)\equiv \sum_{k=1}^N \Omega_k(r,\psi) \mu^k(t),\\
\label{tildeLO}
\widetilde\Lambda_N(r,\psi,S,t)=\mathcal O\left(\mu^{N+1}(t)\right),
 \quad
\widetilde\Omega_N(r,\psi,S,t)=\mathcal O\left(\mu^{N+1}(t)\right), 
\quad 
\sigma_{i,j}(r,\psi,S,t)=\mathcal O\left(\mu^{p}(t)\right)
\end{gather}
as $t\to\infty$ uniformly for all $r\in [\epsilon, \mathcal R-\epsilon]$ and $(\psi,S)\in\mathbb R^2$, where $\Lambda_k(r,\psi)$ and $\Omega_k(r,\psi)$ are $2\pi$-periodic in $\psi$, $\Omega_k(r,\psi)\equiv -\kappa s_k/\varkappa$ for $k<n$, 
\begin{gather*}
\Lambda_n(r,\psi)\equiv \left\langle a_{1,n}\left(r,\frac{\kappa}{\varkappa}S +\psi,S\right)\right\rangle_{\varkappa S}, \quad  \Omega_n(r,\psi) \equiv \left\langle a_{2,n}\left(r,\frac{\kappa}{\varkappa}S +\psi,S\right)\right\rangle_{\varkappa S}-\frac{\kappa s_n}{\varkappa}.
\end{gather*}
\end{Th}
The proof is contained in Section~\ref{sec3}.

Thus, it follows from \eqref{ch1} and \eqref{ch2} that the solutions of system \eqref{rpsi} with asymptotically constant amplitude $r(t)$ correspond to resonant solutions of system \eqref{PS} with $\varrho(t)\approx {\hbox{\rm const}}$. 

Let us remark that the transformation described in Theorem~\ref{Th1} averages with respect to $S$ only the drift terms of equations, while the diffusion coefficients are not simplified. The proposed method is based on the study of the simplified dynamics described by a truncated system 
\begin{gather}\label{trsys}
\frac{d\rho}{dt}=\Lambda(\rho,\phi,S(t),t), \quad \frac{d\phi}{dt}=\Omega(\rho,\phi,S(t),t),
\end{gather}
obtained from \eqref{rpsi} after dropping the stochastic part of the equations. At the second stage, the stochastic stability of this dynamics is proved in the complete system.

It is known that qualitative and asymptotic properties of solutions to asymptotically autonomous systems depend on the corresponding limiting equations (see, for example,~\cite{LM56}). Note that the form of the limiting system corresponding to \eqref{trsys} can vary depending on the properties of perturbations.

Let $q\in [1,N]$ be an integer such that 
\begin{gather}\label{asq}
\begin{split}
&\Omega_k(\rho,\phi)\equiv 0, \quad  k<q, \quad \Omega_q(\rho,\phi)\not\equiv 0.
\end{split}
\end{gather}
From the proof of Theorem~\ref{Th1} it follows that in this case $\Omega_q(\rho,\phi)$ has the following form:
\begin{gather}\label{Omegaq}
\Omega_q(\rho,\phi)\equiv 
\begin{cases} 
\displaystyle-\frac{\kappa s_q}{\varkappa}, & \text{if} \quad q<n,\\
\displaystyle \left\langle a_{2,q}\left(\rho,\frac{\kappa}{\varkappa}S +\phi,S\right)\right\rangle_{\varkappa S}-\frac{\kappa s_q}{\varkappa}, & \text{if } \quad q\geq n.
\end{cases}
\end{gather}
Hence, under assumption \eqref{asq}, the limiting system has the following form:
\begin{gather}\label{limsys}
\frac{d\hat \rho}{dt}=\mu^n(t)\Lambda_n(\hat\rho,\hat\phi), \quad 
\frac{d\hat \phi}{dt}=\mu^q(t)\Omega_q(\hat\rho,\hat\phi).
\end{gather}
Consider the dynamics near fixed points of system \eqref{limsys}. Assume that 
\begin{gather}\label{as1}
\exists\, \rho_0\in(0,\mathcal R), \phi_0\in\mathbb R: \quad \Lambda_n(\rho_0,\phi_0)=0, \quad \Omega_q(\rho_0,\phi_0)=0, \quad \mathcal D_{n,q}\neq 0,
\end{gather}
where 
\begin{gather*}
\mathcal D_{n,q}:=\det {\bf Y}_{n,q}(\rho_0,\phi_0,1), \quad 
{\bf Y}_{n,q}(\rho,\phi,\mu)\equiv \begin{pmatrix}\mu^n\partial_\rho\Lambda_n(\rho,\phi) & \mu^n\partial_\phi\Lambda_n(\rho,\phi) \\ \mu^q\partial_\rho\Omega_q(\rho,\phi)  & \mu^q\partial_\phi\Omega_q(\rho,\phi)
\end{pmatrix}.
\end{gather*}
Note that the case $\mathcal D_{n,q}=0$ corresponds to some bifurcations of the limiting system and is not discussed in this paper. From \eqref{Omegaq} it follows that \eqref{as1} can hold only if $q\geq n$. Define the parameters
\begin{align*}
\lambda_n=\partial_\rho \Lambda_n(\rho_0,\phi_0), \quad 
\xi_n=\partial_\phi \Lambda_n(\rho_0,\phi_0), \quad 
\eta_q=\partial_\rho\Omega_q(\rho_0,\phi_0), \quad 
\omega_q=\partial_\phi\Omega_q(\rho_0,\phi_0).
\end{align*}
Then, the eigenvalues of ${\bf Y}_{n,q}(\rho_0,\phi_0,\mu(t))$ are given by
\begin{gather*}
y_1(t)\equiv \mu^n(t)\beta_1(t), \quad y_2(t)\equiv \mu^q(t)\beta_2(t), \\
\beta_{j}(t)=\begin{cases} \beta_j^0, & \text{if} \quad q=n, \\
\beta_j^0 +\mathcal O(\mu^{q-n}(t)), & \text{if} \quad q>n,
\end{cases}
\end{gather*}
where  
\begin{align*}
& \beta_1^0=
\begin{cases}
\lambda_n, & \text{if} \quad q>n, \\
\frac{1}{2}\left(\lambda_n+\omega_q+\sqrt{(\lambda_n+\omega_q)^2-4\mathcal D_{n,q}}\right), & \text{if} \quad q=n,
\end{cases}
\\ 
& \beta_2^0=
\begin{cases}
\lambda_n^{-1} \mathcal D_{n,q}, & \text{if} \quad q>n, \\
\frac{1}{2}\left(\lambda_n+\omega_q-\sqrt{(\lambda_n+\omega_q)^2-4\mathcal D_{n,q}}\right), & \text{if} \quad q=n.
\end{cases}
\end{align*}
Define
\begin{gather}\label{zetah}
\gamma_k(t):=\int\limits_{t_0}^t \mu^k(\varsigma)\, d\varsigma, \quad , \quad 
\tilde \beta_2^0=\beta_2^0-\delta_{m,q}\frac{(q-n)\chi_m}{2}.
\end{gather}
We have the following lemma.
\begin{Lem}\label{Lem1}
Let assumption \eqref{as1} hold with $q\geq n$ and $\gamma_q(t)\to\infty$ as $t\to\infty$.
\begin{itemize}
\item If $q=n$ and 
\begin{itemize}
\item $\Re \beta^0_j<0$ for all $j\in\{1,2\}$, then the equilibrium $(\rho_0,\phi_0)$ of system \eqref{limsys} is asymptotically stable;
\item $\Re \beta^0_j>0$ for some $j\in\{1,2\}$, then the equilibrium $(\rho_0,\phi_0)$ of system \eqref{limsys} is unstable.
\end{itemize}
\item If $q>n$ and 
\begin{itemize}
\item $\beta^0_1<0$, $\tilde\beta^0_2<0$, then the equilibrium $(\rho_0,\phi_0)$ of system \eqref{limsys} is asymptotically stable;
\item $\beta^0_1>0$, $\tilde\beta^0_2>0$, then the equilibrium $(\rho_0,\phi_0)$ of system \eqref{limsys} is unstable.
\end{itemize}
\end{itemize}
\end{Lem}

Consider the case, when the equilibrium $(\rho_0,\phi_0)$ of system \eqref{limsys} is stable. Let us show that this dynamics is preserved in the truncated system \eqref{trsys}. In particular, we have the following lemma.

\begin{Lem}\label{Lem2}
Let assumptions \eqref{asq}, \eqref{as1} hold with $n\leq q<  (2N+2+n)/3$. If $\Re\beta_1^0<0$ and $\Re\tilde\beta_2^0<0$, then system \eqref{trsys} has a stable particular solution $\rho_{L}(t)$, $\phi_{L}(t)$ as $t\geq t_0$ such that
\begin{gather*}
\rho_{L}(t)=\rho_\ast(t)+\mathcal O\left(\mu^{\alpha}(t)\right) , \quad 
\phi_{L}(t)=\phi_\ast(t)+\mathcal O\left(\mu^{\alpha+\frac{q-n}{2}}(t)\right), \quad t\to\infty
\end{gather*}
with some $\alpha\in (0,\alpha_0)$, where
\begin{gather}
  \label{srp}
\rho_\ast(t)\equiv \rho_0+\sum_{k=1}^{N-q} \rho_k \mu^k(t), \quad 
\phi_\ast(t)\equiv \phi_0+\sum_{k=1}^{N-q}\phi_k  \mu^k(t),\\
\nonumber
\alpha_0=
	 \frac{1}{2}\min\left\{\frac{2\Re \beta_1^0}{\chi_m}, \frac{2\Re\tilde\beta_2^0}{\chi_m},2N+2+n-3q\right\},
\end{gather}
and $\rho_k$, $\phi_k$ are some constants.
  Moreover, the solution $\rho_{L}(t)$, $\phi_{L}(t)$ is asymptotically stable if $\gamma_q(t)\to\infty$ as $t\to\infty$.
\end{Lem}

Now, let us show that the phase locking regime associated with the solution $\rho_L(t)$, $\phi_L(t)$ of system \eqref{trsys} is preserved in the complete system \eqref{rpsi}. Consider the function
\begin{gather*}
\mathcal M_L(r,\psi,t):= \sqrt{\left(r-\rho_L(t)\right)^2 +\left(\psi-\phi_L(t)\right)^2 \mu^{n-q}(t)}
\end{gather*}
defined for all $(r,\psi)\in\mathbb R^2$ and $t\geq t_0$. 
We have the following theorem.
\begin{Th}\label{Th2}
Let system \eqref{PS} satisfy \eqref{rc}, \eqref{fgas}, \eqref{mucond}, and assumptions \eqref{asq}, \eqref{as1} hold with $n\leq q<  (2N+2+n)/3$ and $p>(q-n)/2$. If $\Re\beta_1^0<0$ and $\Re\tilde\beta_2^0<0$, then for all $\epsilon_1>0$, $\epsilon_2>0$, $l\in (0,1)$ and $t_s>t_0$ there exist $\delta_1>0$ and $\delta_2>0$ such that for any $0<\varepsilon\leq \delta_2$ the solution $r(t)$, $\psi(t)$ of system \eqref{rpsi} with initial data
$ \mathcal M_L(r(t_s),\psi(t_s),t_s)\leq \delta_1$  satisfies
\begin{gather}\label{rpst}
\mathbb P\left(\sup_{0<t-t_s\leq \mathcal T} \mathcal M_L(r(t),\psi(t),t)\geq \epsilon_1\right)\leq \epsilon_2,
\end{gather}
where
\begin{gather*}
\mathcal T=
\begin{cases}
\infty, & \text{if} \quad \mu^{2p-q+n}(t)\in L_1(t_0,\infty),\\
\mathcal T_\varepsilon, & \text{if} \quad  \mu^{2p-q+n}(t)\not\in L_1(t_0,\infty)
\end{cases}
\end{gather*}
and $\mathcal T_\varepsilon>0$ is the root of the equation
\begin{gather*}
\gamma_{2p-q+n}(\mathcal T_\varepsilon+t_s)-\gamma_{2p-q+n}(t_s)=\varepsilon^{-2(1-l)}.
\end{gather*} 
\end{Th}

Note that the estimates of the form \eqref{rpst} are usually associated with the concept of stability in probability (see, for example,~\cite[\S 5.3]{RH12} and \cite[Chapter II, \S 1]{HJK67}). 

The proofs of Lemmas~\ref{Lem1}, \ref{Lem2} and Theorem~\ref{Th2} are contained in Section~\ref{sec4}.

Now, consider the case when instead of \eqref{as1} the following assumption holds:
\begin{gather}\label{as2}
	\begin{split}
\exists\, \rho_0\in(0,\mathcal R): \quad & \Lambda_n(\rho_0,\phi)\equiv 0, \quad  \Xi_{n}(\phi):=\partial_r\Lambda_n(\rho_0,\phi)\not\equiv 0,\\ & \Omega_q(\rho_0,\phi)\neq 0 \quad \forall\,\phi\in\mathbb R.
	\end{split}
\end{gather}
In this case, the solutions with $\rho(t)\approx \rho_0$ may appear in the truncated system \eqref{trsys} in the phase drift mode, when the phase shift tends to infinity. In particular, we have the following lemma.
\begin{Lem}\label{Lem3}
Let assumptions \eqref{asq}, \eqref{as2} hold and $\gamma_q(t)\to \infty$ as $t\to\infty$. If $\sup_\psi\Xi_n(\psi)<\delta_{m,n}\chi_m/2$, then system \eqref{trsys} has a solution $\tilde\rho_D(t)$, $\tilde\phi_D(t)$ as $t\geq t_0$ such that 
\begin{gather*}
\tilde\rho_D(t)=\rho_0+\mathcal O\left(\mu^{\frac{1}{2}}(t)\right), \quad 
\tilde\phi_D(t) \to \infty, \quad t\to\infty
\end{gather*}
\end{Lem}

Note that assumption \eqref{as2} can be relaxed if $n>q$. In this case, $\Omega_q(\rho_0,\phi)\equiv -\kappa s_q/\varkappa$, and system \eqref{rpsi} can be further simplified by additionally averaging the amplitude equation over the fast variable $\psi(t)$. 
We have the following lemma.
\begin{Lem}\label{Lem4}
Let assumptions \eqref{asq} hold with $n>q$, $2p>q$ and $s_q\neq 0$. Then, for all $N\in [n,m]$ and $\epsilon\in (0,\mathcal R/2)$ there exist $t_1\geq t_0$ and the reversible transformation $(r,\psi)\mapsto (z,\psi)$,
\begin{align}
\label{ch3} z(t)=r(t)+\tilde E_N(r(t),\psi(t),t),
\end{align}
where $|\tilde E_N(r,\psi,t)|\leq \epsilon$ for all $r\in (0,\mathcal R]$, $\psi\in\mathbb R$, $t\geq t_1$, and $\tilde E_N(r,\psi,t)=\mathcal O(\mu^{n-q}(t))$ as $t\to\infty$ such that system \eqref{rpsi} can be transformed into
\begin{gather}\label{rpsi2}
d\begin{pmatrix}
z \\ \psi
\end{pmatrix} = \begin{pmatrix} \mathcal F(z,\psi,t) \\ 
\mathcal G(z,\psi,t)\end{pmatrix}dt+\varepsilon {\bf Z}(z,\psi,t)\, d{\bf w}(t),
\end{gather}
where  
\begin{align*}
\mathcal F(z,\psi,t)\equiv&\,\widehat{\mathcal F}_N(z,t)+ \widetilde{\mathcal F}_N(z,\psi,t), \quad  & \widehat{\mathcal F}_N(z,t)\equiv &\, \sum_{k=n}^N {\mathcal F}_k(z) \mu^k(t),\\
\mathcal G(z,\psi,t)\equiv&\, -\frac{\kappa s_q}{\varkappa} \mu^q(t) + \widetilde{\mathcal G}_q(z,\psi,t), &&
\end{align*}
and ${\bf Z}(z,\psi,t)\equiv \{\zeta_{i,j}(z,\psi,t)\}_{2\times 2}$ are defined for all $z\in (0, \mathcal R]$, $\psi\in\mathbb R$, $t\geq t_1$, are $2\pi$-periodic w.r.t. $\psi$, and
\begin{gather}
\label{tildeLO2}
\widetilde{\mathcal F}_N(z,\psi,t)=\mathcal O\left(\mu^{N+1}(t)\right),
 \quad
\widetilde{\mathcal G}_q(z,\psi,t)=\mathcal O\left(\mu^{q+1}(t)\right), 
\quad 
\zeta_{i,j}(z,\psi,t)=\mathcal O\left(\mu^{p}(t)\right)
\end{gather}
as $t\to\infty$ uniformly for all $z\in [\epsilon, \mathcal R-\epsilon]$ and $\psi\in\mathbb R$. In particular,  $\mathcal F_n(z)\equiv \left\langle \Lambda_n(z,\psi)\right\rangle_{\psi}$.
\end{Lem}

This time, consider another truncated system 
\begin{gather}\label{trsys2}
\frac{d\rho}{dt}=\mathcal F(\rho,\phi,t), \quad 
\frac{d\phi}{dt}=\mathcal G(\rho,\phi,t),
\end{gather}
obtained from \eqref{rpsi2}, and assume that
\begin{gather}\label{as22}
	\begin{split}
\exists\, \rho_0\in(0,\mathcal R): \quad & \mathcal F_n(\rho_0)=0, \quad  \hat\xi_{n}:=\partial_r\mathcal F_n(\rho_0)\neq 0.
	\end{split}
\end{gather}
Then, we have the following lemma similar to Lemma~\ref{Lem3}.

\begin{Lem}\label{Lem5}
Let assumptions \eqref{asq}, \eqref{as22} hold with $n>q$, $2p>q$, $s_q\neq 0$ and $\gamma_q(t)\to \infty$ as $t\to\infty$. If $\hat \xi_n<\delta_{m,n}\chi_m/2$, then system \eqref{trsys2} has a solution $\rho_D(t)$, $\phi_D(t)$ as $t\geq t_0$ such that 
\begin{gather*}
\rho_D(t)=\rho_0+\mathcal O\left(\mu^{\frac{1}{2}}(t)\right), \quad 
\phi_D(t) =-\frac{\kappa s_q }{\varkappa}\gamma_q(t) (1+o(1)), \quad t\to\infty
\end{gather*}
\end{Lem}

Finally, let us show that the solutions with $r(t)\approx \rho_0$ are preserved in the stochastic system \eqref{rpsi} in the phase drift mode.
Define $\mathcal M_D(r,t):=|r-\rho_D(t)|$. Then, we have the following theorem.
\begin{Th}\label{Th3}
Let system \eqref{PS} satisfy \eqref{rc}, \eqref{fgas}, \eqref{mucond}, assumptions \eqref{asq}, \eqref{as22} hold with $n>q$, $2p>q$, $s_q\neq 0$, and $\gamma_q(t)\to\infty$ as $t\to\infty$. If $\hat\xi_n<\delta_{m,n}\chi_m/2$, then there is $t_\ast>t_0$ such that for all $\epsilon_1>0$, $\epsilon_2>0$, $l\in (0,1)$ and $t_s>t_\ast$ there exist $\delta_1>0$ and $\delta_2>0$ such that for any $0<\varepsilon\leq \delta_2$ the solution $r(t)$, $\psi(t)$ of system \eqref{rpsi} with initial data
$ \mathcal M_D(r(t_s),t_s)\leq \delta_1$, $\psi(t_s)\in\mathbb R$  satisfies
\begin{gather}\label{rpD}
\mathbb P\left(\sup_{0<t-t_s\leq \mathcal T} \mathcal M_D(r(t),t)\geq \epsilon_1\right)\leq \epsilon_2,
\end{gather}
where
\begin{gather*}
\mathcal T=
\begin{cases}
\infty, & \text{if} \quad \mu^{c}(t)\in L_1(t_0,\infty),\\
\mathcal T_\varepsilon, & \text{if} \quad  \mu^{c}(t)\not\in L_1(t_0,\infty),
\end{cases}
\end{gather*}
$c=\min\{2p,n+1\}$, and $\mathcal T_\varepsilon>0$ is the root of the equation
\begin{gather*} 
\gamma_{c}(\mathcal T_\varepsilon+t_s)-\gamma_{c}(t_s)=\varepsilon^{-2(1-l)}.
\end{gather*} 
\end{Th}

Note that the estimate \eqref{rpD} corresponds to partial stability in probability (see~\cite[\S 7.1.2]{VIV98}).

The proofs of Lemmas~\ref{Lem3}, \ref{Lem4}, \ref{Lem5} and Theorem~\ref{Th3} are contained in Section~\ref{sec5}.

\section{Change of variables} \label{sec3}

Applying It\^{o}'s formula (see, for example,~\cite[\S 4.2]{BO98}) to \eqref{ch1} yields the following system:
\begin{gather}\label{RPsys}
d\begin{pmatrix}
R \\ \Psi
\end{pmatrix}=
{\bf b}(R,\Psi,S(t),t)\, dt + \varepsilon{\bf B}(R,\Psi,S(t),t)\,d{\bf w}(t),
\end{gather}
where ${\bf b}(R,\Psi,S,t)\equiv (b_1(R,\Psi,S,t),b_2(R,\Psi,S,t))^T$, ${\bf B}(R,\Psi,S,t)\equiv \{\beta_{i,j}(R,\Psi,S,t)\}_{2\times 2}$, 
\begin{align*}
b_1(R,\Psi,S,t) &\equiv
 a_1\left(R,\frac{\kappa}{\varkappa}S +\Psi,S,t\right), \\
 b_2(R,\Psi,S,t) & \equiv a_2\left(R,\frac{\kappa}{\varkappa}S +\Psi,S,t\right)-\frac{\kappa}{\varkappa}S'(t), \\
\beta_{i,j}(R,\Psi,S,t) & \equiv   \alpha_{i,j}\left(R,\frac{\kappa}{\varkappa}S +\Psi,S,t\right).
\end{align*}
It follows from \eqref{fgas} that
\begin{gather}\label{FGas}\begin{split}
{\bf b}(R,\Psi,S(t),t)&= \sum_{k=1}^{M}  {\bf b}_k(R,\Psi,S(t)) \mu^{k}(t)+\tilde {\bf b}_{M}(R,\Psi,t), \\ 
{\bf B}(R,\Psi,S(t),t)&= \sum_{k=p}^{p+M} {\bf B}_k(R,\Psi,S(t))\mu^{k}(t)+\tilde {\bf B}_{p+M}(R,\Psi,t), 
\end{split}
\end{gather}
where 
\begin{gather}\label{tildeas}
\begin{split}
& |\tilde {\bf b}_{M}(R,\Psi,t)|=\mathcal O\left(\mu^{M+1}(t)\right), \quad  \|\tilde {\bf B}_{p+M}(R,\Psi,t)\|=\mathcal O\left(\mu^{p+M+1}(t)\right)
\end{split}
\end{gather}
as  $t\to\infty$ uniformly for all $R\in(0,\mathcal R]$, $\Psi\in\mathbb R$ and for any integer $M\geq 1$. The coefficients ${\bf b}_k(R,\Psi,S)\equiv (b_{1,k}(R,\Psi,S),b_{2,k}(R,\Psi,S))^T$ and ${\bf B}_k(R,\Psi,S)\equiv \{\beta_{i,j,k}(R,\Psi,S)\}_{2\times 2}$ are $2\pi$-periodic in $\Psi$ and $2\pi \varkappa$-periodic in $S$.  In particular, $b_{1,k}(R,\Psi,S) \equiv 0$ and $b_{2,k}(R,\Psi,S)\equiv - \kappa s_{k}/\varkappa$ for $k<n$.

Let us show that system \eqref{RPsys} can be simplified by averaging the drift terms of equations. Such method is effective in the study of similar deterministic systems with a small parameter~\cite{BM61,Hap93}.
It is readily seen that the limiting system corresponding to \eqref{RPsys} has the following form:
\begin{gather*}
\frac{d\hat R}{dt}=0, \quad \frac{d\hat\Psi}{dt}=0, \quad \frac{d\hat S}{dt}=s_0.
\end{gather*}
Hence, the phase of the perturbations $S(t)$ can be used as a fast variable as $t\to\infty$. The averaging transformation is constructed in the following form:
\begin{gather}\label{rpch}
\begin{split}
	U_N(R,\Psi,S,t)&=R+\sum_{k=1}^N  u_k(R,\Psi,S) \mu^{k}(t), \\
	V_N(R,\Psi,S,t)&=\Psi+\sum_{k=1}^N v_k(R,\Psi,S) \mu^{k}(t)
\end{split}
\end{gather}
with some integer $N\in [n,m]$ and some functions $u_k(R,\Psi,S)$ and $v_k(R,\Psi,S)$, which are assumed to be periodic with respect to $\Psi$ and $S$. These coefficients are chosen in such that the system written in the new variables 
$
r(t)\equiv U_N(R(t),\Psi(t),S(t),t)$, $
\psi(t)\equiv V_N(R(t),\Psi(t),S(t),t)
$
takes the form \eqref{rpsi}, where the drift terms do not depend on $S(t)$ in the first $N$ terms of the asymptotics as $t\to\infty$. 

Define the operators
\begin{gather}\nonumber
\mathcal L U:= 	 \partial_t U  + \left(\nabla_{(R,\Psi)} U\right)^T {\bf b}+\frac{\varepsilon^2}{2}{\hbox{\rm tr}}\left({\bf B}^T {\bf H}_{(R,\Psi)}(U){\bf B}\right), \\ 
\label{HJdef}
\nabla_{(R,\Psi)} U :=  	\begin{pmatrix}\partial_{R} U  \\  \partial_{\Psi} U \end{pmatrix}, \quad 
	{\bf H}_{(R,\Psi)}(U):=  		
\begin{pmatrix}
		\partial_{R}^2 U & \partial_{R}\partial_{\Psi} U \\
		\partial_{\Psi}\partial_{R} U & \partial_{\Psi}^2 U
\end{pmatrix}, \quad
{\bf J}_{(R,\Psi)}(U,V):=  		
\begin{pmatrix}
		\partial_{R} U & \partial_{\Psi} U \\
		\partial_{R} V & \partial_{\Psi} V
\end{pmatrix}
\end{gather}
for any smooth functions $U(R,\Psi,t)$ and $V(R,\Psi,t)$. Note that $\mathcal L$ is the generator of the process defined by system \eqref{RPsys} (see, for example,~\cite[\S 3.3]{RH12}). By applying It\^{o}'s formula, we get
\begin{equation}\label{Itorhopsi}
\begin{array}{rcl}
d \begin{pmatrix} r \\ \psi \end{pmatrix} &=&\displaystyle \begin{pmatrix} \mathcal L  U_N( R,\Psi,S(t),t) \\ \mathcal L  V_N( R,\Psi,S(t),t) \end{pmatrix}  dt \\
&&\\
&&+ \varepsilon {\bf J}_{(R,\Psi)}(U_N( R,\Psi,S(t),t),V_N( R,\Psi,S(t),t)) {\bf B}( R,\Psi,S(t),t) \, d{\bf w}(t).
\end{array}
\end{equation}
Taking into account \eqref{FGas} with $M=N$, we see that the drift term in \eqref{Itorhopsi} takes the form
\begin{gather}\label{drpas}
\begin{split}
\begin{pmatrix} \mathcal L U_N  \\ \mathcal L V_N \end{pmatrix} = &
\sum_{k=1}^N \mu^k(t)\left\{{\bf b}_k+s_0\partial_S \begin{pmatrix} u_k \\ v_k \end{pmatrix}\right\} + 
\sum_{k=2}^N \mu^k(t) \sum_{j=1}^{k-1}\left\{ b_{1,j}\partial_R + b_{2,j}\partial_\Psi+s_{j} \partial_S \right\}\begin{pmatrix} u_{k-j} \\ v_{k-j} \end{pmatrix}\\
 & +\frac{\varepsilon^2}{2}\sum_{k=2p+1}^N \mu^k(t) \sum_{
        \substack{
             i+j+l=k\\
            i\geq p, j\geq p, l\geq 1
                }
            } 
\begin{pmatrix}
{\hbox{\rm tr}}\left \{{\bf B}^T_{i} {\bf H}_{(R,\Psi)}(u_{l} ){\bf B}_{j}\right\} \\ 
{\hbox{\rm tr}}\left\{{\bf B}^T_{i} {\bf H}_{(R,\Psi)}(v_{l} ){\bf B}_{j}\right\} \end{pmatrix} + \tilde{\bf z}_N(R,\Psi,t),
\end{split}
\end{gather}
where  
\begin{gather}\label{zpas}\begin{split}
\tilde{\bf z}_N(R,\Psi,t)\equiv & \, \tilde {\bf b}_N(R,\Psi,t)+
\sum_{k=N+1}^{2N} \mu^k(t) \sum_{j=1}^{k-1}\left\{ b_{1,j}\partial_R + b_{2,j}\partial_\Psi+s_j \partial_S \right\}\begin{pmatrix} u_{k-j} \\ v_{k-j} \end{pmatrix}\\
 & +\frac{\varepsilon^2}{2}\sum_{k= N+1}^{2p+3N} \mu^k(t) \sum_{
        \substack{
            i+j+l=k\\
            i\geq p, j\geq p, l\geq 1
                }
            } 
\begin{pmatrix}
{\hbox{\rm tr}}\left ({\bf B}^T_{i} {\bf H}_{(R,\Psi)}(u_{l} ){\bf B}_{j}\right) \\ 
{\hbox{\rm tr}}\left({\bf B}^T_{i} {\bf H}_{(R,\Psi)}(v_{l} ){\bf B}_{j}\right) \end{pmatrix}\\
 & +\frac{\varepsilon^2}{2}\sum_{k= 1}^{N} \mu^k(t) 
\begin{pmatrix}
{\hbox{\rm tr}}\left (\tilde {\bf B}^T_{p+N} {\bf H}_{(R,\Psi)}(u_{k} )\tilde{\bf B}_{p+N}\right) \\ 
{\hbox{\rm tr}}\left(\tilde{\bf B}^T_{p+N} {\bf H}_{(R,\Psi)}(v_{k} )\tilde{\bf B}_{p+N}\right) \end{pmatrix} \\
& +\sum_{k=1}^N \mu^k(t) \left(k\ell(t)+\tilde b_{1,N}\partial_R+\tilde b_{2,N}\partial_\Psi\right)\begin{pmatrix} u_{k} \\ v_{k} \end{pmatrix}.
\end{split}
\end{gather}
Combining \eqref{zpas} with \eqref{mucond}, \eqref{FGas} and \eqref{tildeas}, we get
\begin{gather}\label{tildeZas}
|\tilde {\bf z}_{N}(R,\Psi,t)|=\mathcal O\left(\mu^{N+1}(t)\right)
\end{gather}
 as $t\to\infty$ uniformly for all $R\in (0,\mathcal R]$ and $\Psi\in\mathbb R$. It is assumed that $u_k(R,\Psi,S)\equiv v_k(R,\Psi,S)\equiv 0$ for $k\leq 0$ and $k>N$. Comparing the drift terms of system \eqref{rpsi} with \eqref{drpas}, we obtain
\begin{gather}\label{ukvk}
s_0\partial_S \begin{pmatrix} u_k \\ v_k \end{pmatrix}=\begin{pmatrix} \Lambda_k(R,\Psi) \\ \Omega_k(R,\Psi)\end{pmatrix}-{\bf b}_k(R,\Psi,S)+\tilde{\bf h}_k(R,\Psi,S), \quad k=1,\dots, N,
\end{gather} 
where each $\tilde {\bf h}_k(R,\Psi,S)\equiv (\tilde h_{1,k}(R,\Psi,S),\tilde h_{2,k}(R,\Psi,S))^T$ is expressed in terms of $\{u_{j}, v_{j},\Lambda_{j}, \Omega_{j}\}_{j=1}^{k-1}$. In particular, if $p=n=1$, then
\begin{equation*}
\begin{array}{rl}
\tilde {\bf h}_1 \equiv 
& \begin{pmatrix}
0\\ 0
\end{pmatrix},\\
\tilde {\bf h}_2 \equiv 
& 
(u_1\partial_R+v_1\partial_\Psi)
\begin{pmatrix}
\Lambda_1\\ \Omega_1
\end{pmatrix}
-(b_{1,1}\partial_R + b_{2,1}\partial_\Psi )\begin{pmatrix} u_{1} \\ v_{1} \end{pmatrix}, \\
\tilde {\bf h}_3 \equiv
& \displaystyle
\sum_{i+j=3}(u_i\partial_R+v_i\partial_\Psi)
\begin{pmatrix}
\Lambda_j\\ \Omega_j
\end{pmatrix} + \frac{1}{2}\left(u_1^2\partial_R^2+2u_1v_1 \partial_R\partial_\Psi+v_1^2 \partial_\Psi^2\right)\begin{pmatrix}
\Lambda_1\\ \Omega_1
\end{pmatrix} \\
& \displaystyle -\sum_{j=1}^2 \left( b_{1,j}\partial_R + b_{2,j}\partial_\Psi+s_j\partial_S\right)\begin{pmatrix} u_{3-j} \\ v_{3-j} \end{pmatrix} -\frac{\varepsilon^2}{2}\begin{pmatrix}
{\hbox{\rm tr}}\left ({\bf B}^T_{1} {\bf H}_{(R,\Psi)}(u_{1} ){\bf B}_{1}\right) \\ 
{\hbox{\rm tr}}\left({\bf B}^T_{1} {\bf H}_{(R,\Psi)}(v_{1} ){\bf B}_{1}\right) \end{pmatrix}.
\end{array}
\end{equation*}
Let us define 
\begin{gather*}
\begin{pmatrix}
\Lambda_k(R,\Psi) \\
\Omega_k(R,\Psi)
\end{pmatrix} \equiv \left\langle {\bf b}_k(R,\Psi,S)-\tilde {\bf h}_k(R,\Psi,S)\right\rangle_{\varkappa S}.
\end{gather*}
In this case, the right-hand side of \eqref{ukvk} has zero mean. Hence, system \eqref{ukvk} is solvable in the class of functions that are $2\pi\varkappa$-periodic in $S$ with zero mean $\langle u_k(R,\Psi,S)\rangle_{ \varkappa S}=\langle v_k(R,\Psi,S)\rangle_{ \varkappa S}=0$.
By induction, we can see that the functions $\Lambda_k(R,\Psi)$ and $\Omega_k(R,\Psi)$ are $2\pi$-periodic in $\Psi$. It can easily be checked that 
\begin{align*}
&\Lambda_k(R,\Psi)\equiv 0, \quad &\Omega_k(R,\Psi)&\equiv \langle b_{2,k}(R,\Psi,S)\rangle_{\varkappa S}\equiv -\frac{\kappa s_k}{\varkappa}, \quad k< n,\\
&\Lambda_n(R,\Psi)\equiv \langle b_{1,n}(R,\Psi,S)\rangle_{\varkappa S}, \quad &\Omega_n(R,\Psi)&\equiv \langle b_{2,n}(R,\Psi,S)\rangle_{\varkappa S}.
\end{align*}
From \eqref{rpch} it follows that for all $\epsilon\in (0,\mathcal R/2)$ there exists $t_0\geq \tau_0$ such that
\begin{gather*}
|U_N(R,\Psi,S,t)-R|< \epsilon, \quad 
|V_N(R,\Psi,S,t)-\Psi|< \epsilon, \quad 
|D_N(R,\Psi,S,t)-1|< \epsilon
\end{gather*}
for all $R\in(0,\mathcal R]$, $(\Psi,S)\in\mathbb R^2$ and $t\geq t_0$, where 
\begin{gather*}
D_N(R,\Psi,S,t):={\hbox{\rm det }} {\bf J}_{(R,\Psi)}(U_N( R,\Psi,t),V_N( R,\Psi,t)).
\end{gather*} 
Hence, there exists the inverse transformation $R= u(r,\psi,t)$, $\Psi=v(r,\psi,t)$ such that $0<u(r,\psi,t)<\mathcal R$ for all $r\in [\epsilon,\mathcal R-\epsilon]$, $\psi\in\mathbb R$ and $t\geq t_0$.
Define
\begin{align}\nonumber
\begin{pmatrix}
\tilde \Lambda_N(r,\psi,t)\\
\tilde \Omega_N(r,\psi,t)
\end{pmatrix}
\equiv & \,
\begin{pmatrix}
\mathcal L U_N (R,\Psi,S(t),t)\\
\mathcal L V_N(R,\Psi,S(t),t)
\end{pmatrix}\Big|_{\substack{R=u(r, \psi, t)  \\ \Psi=v(r,\psi,t)}}  - \sum_{k=1}^N  \begin{pmatrix}
\Lambda_k(r,\psi)\\
\Omega_k(r,\psi)
\end{pmatrix} \mu^k(t), \\
\label{Cform}
{\bf C}(r,\psi,S(t),t)\equiv & \,
{\bf J}_{(R,\Psi)} (U_N( R,\Psi,S(t),t),V_N( R,\Psi,S(t),t)) {\bf B}(R,\Psi,S(t),t)\Big|_{\substack{ R=u(r, \psi, t) \\  \Psi=v(r,\psi,t)}}.
\end{align}
Combining this with \eqref{drpas}, \eqref{tildeZas} and \eqref{FGas}, we get estimates \eqref{tildeLO}.

Define $\tilde U_N(R,\Psi,t)\equiv  U_N(R,\Psi,S(t),t)-R$, $\tilde V_N(R,\Psi,t)\equiv  V_N(R,\Psi,S(t),t)-\Psi$. Then, we get \eqref{tildeUNVN}. 
This completes the proof of Theorem~\ref{Th1}.

\section{Analysis of the phase locking mode}
\label{sec4}

\begin{proof}[Proof of Lemma~\ref{Lem1}]
Let $q=n$. The change of variables $\hat\rho(t)=\hat z_1(\gamma_n(t))$, $\hat\psi(t)=\hat z_2(\gamma_n(t))$ transforms the limiting system \eqref{limsys} into the following autonomous system: 
\begin{gather*}
\frac{d\hat z_1}{d\gamma}= \Lambda_n(\hat z_1,\hat z_2), \quad \frac{d\hat z_2}{d\gamma}= \Omega_q(\hat z_1,\hat z_2).
\end{gather*}
By applying Lyapunov's indirect method (see, for example,~\cite[Theorem 4.7]{HKH}), we obtain the result of the lemma.  

Now, let $q>n$. Substituting 
\begin{gather*}
\hat\rho(t)=\rho_0+z_1(t),\quad 
\hat\phi(t)=\phi_0+\mu^{\frac{q-n}{2}}(t)z_2(t)
\end{gather*} 
into \eqref{limsys}, we obtain
\begin{gather}\label{zsys}
 \frac{d{\bf z}}{dt}= {\bf f}({\bf z},t), 
\end{gather}
where ${\bf z}=(z_1,z_2)^T$, ${\bf f}({\bf z},t)=(f_1({\bf z},t),f_2({\bf z},t))^T$, 
\begin{align*}
f_1({\bf z},t) \equiv &\,  
\mu^n(t)\Lambda_n\left(\rho_0+z_1,\phi_0+\mu^{\frac{q-n}{2}}(t) z_2\right), \\ 
f_2({\bf z},t) \equiv &\, \mu^{\frac{q+n}{2}}(t)\Omega_q\left(\rho_0+z_1,\phi_0+\mu^{\frac{q-n}{2}}(t) z_2\right) -\ell(t)\frac{(q-n)z_2}{2}.
\end{align*}
It follows that 
\begin{align*}
f_1({\bf z},t) = &\,  
\mu^n(t)  \lambda_n z_1 \left(1 +\mathcal O(z_1 ) \right)+  \mu^{\frac{q+n}{2}}(t) \left( \xi_n z_2 + \mathcal O\left(|{\bf z}|^2\right)\right), \\ 
f_2({\bf z},t) = &\, \mu^{\frac{q+n}{2}}(t)\eta_q z_1 \left( 1+\mathcal O(z_1)\right) +\mu^{q}(t)\left( \tilde\omega_q z_2 \left(1+\mathcal O(\tilde\ell(t))+\mathcal O(\mu(t))\right)+ \mathcal O\left(|{\bf z}|^2\right)\right)
\end{align*}
as $|{\bf z}|=\sqrt{z_1^2+z_2^2}\to 0$ and $t\to\infty$, where
$\tilde \omega_{q}:=\omega_q-\delta_{m,q}{(q-n)\chi_m}/{2}$.
Consider 
\begin{gather*}
\mathcal V(z_1,z_2,t)= B_1 z_1^2+ z_2^2+ \mu^{\frac{q-n}{2}}(t) B_2 z_1z_2 
\end{gather*}
with some parameters $B_1$ and $B_2$ as a Lyapunov function candidate for system \eqref{zsys}. The derivative of $\mathcal V(z_1,z_2,t)$ along the 
trajectories of system \eqref{zsys} is given by 
\begin{align*}
\frac{d\mathcal V}{dt}=& \, 
\mu^n(t)z_1^2\left(2B_1 \lambda_n  +\mathcal O\left(|{\bf z}|\right)+\mathcal O(\mu(t))\right) \\
 & + \mu^q(t) z_2^2 \left( B_2 \xi_n +2\tilde \omega_q+\mathcal O(z_2)+\mathcal O\left(\mu(t) +|\tilde \ell(t)|\right)  \right)  \\
 & + \mu^{\frac{q+n}{2}}(t)z_1z_2\left(2B_1\xi_n+2 \eta_q+B_2\lambda_n +\mathcal O\left(|{\bf z}|\right)+\mathcal O\left(\mu(t) +|\tilde \ell(t)|\right) \right) 
\end{align*}
as $|{\bf z}|\to 0$ and $t\to\infty$. Let us take   
\begin{gather*}
B_1=
\begin{cases}
\lambda_n \tilde\omega_q, & \text{if} \quad \xi_n=0,\\
\displaystyle \frac{\lambda_n \tilde \beta_2^0}{2\xi_n^2}, & \text{if} \quad \xi_n\neq 0,
\end{cases}
\quad  
B_2=-\frac{2 B_1\xi_n+2\eta_q}{\lambda_n}.
\end{gather*}
Then,
\begin{gather*}
\frac{d\mathcal V}{dt} = 2B_1 \beta_1^0  \left( \mu^n(t) z_1^2 + \mu^q(t) (\xi_n^2+\delta_{\xi_n,0})\lambda_n^{-2}z_2^2 \right) \left (1+\mathcal O\left(|{\bf z}|\right)+\mathcal O\left(\mu(t)+|\tilde \ell(t)|\right) \right)  
\end{gather*}
as $|{\bf z}|\to 0$ and $t\to\infty$. 
Note that $B_1>0$ if $\beta_1^0 \tilde \beta_2^0>0$. Therefore, there exist $\Delta_1>0$ and $t_1\geq t_0$ such that
\begin{gather*}
B_- |{\bf z}|^2\leq \mathcal V(z_1,z_2,t)\leq B_+  |{\bf z}|^2
\end{gather*}
and
\begin{align}
\label{vleq} \frac{d\mathcal V}{dt} & \leq -C_0\mu^q(t)  \mathcal V, & \text{if} & \quad \beta_1^0<0, \quad \tilde \beta_2^0<0,\\
\label{vgeq} \frac{d\mathcal V}{dt} & \geq  C_0\mu^q(t)  \mathcal V, & \text{if}& \quad \beta_1^0>0, \quad \tilde \beta_2^0>0
\end{align}
for all $|{\bf z}|\leq \Delta_1$ and $t\geq t_1$ with  $B_+=2\max\{1,B_{1}\}>0$, $B_-=\min\{1,B_{1}\}/2>0$ and $C_0=B_1|\beta_1^0|\min\{1,(\xi_n^2+\delta_{\xi_n,0})\lambda_n^{-2}\}/B_+>0$. 

Let $\beta_1^0<0$ and $\tilde\beta_2^0<0$, then for all $\epsilon\in (0,\Delta_1)$ there exits $\delta_\epsilon=(\epsilon/2)\sqrt{B_-/B_+}$ such that 
\begin{gather*} 
\sup_{|{\bf z}|\leq \delta_\epsilon}\mathcal V(z_1,z_2,t)\leq B_+ \delta_\epsilon^2 < B_-\epsilon^2=\inf_{|{\bf z}|=\epsilon}\mathcal V(z_1,z_2,t).
\end{gather*}
Hence, any solution $(z_1(t),z_2(t))$ of system \eqref{zsys} with initial data $|{\bf z}(t_s)|\leq \delta_\epsilon$ with any $t_s\geq t_1$ cannot exit from the domain $|{\bf z}|\leq\epsilon$ as $t> t_s$. Thus, the equilibrium is stable. Moreover, integrating \eqref{vleq}, we obtain
\begin{gather*}
\left|{\bf z}(t)\right|\leq C_+ \exp\left(-\frac{C_0}{2} (\gamma_q(t)-\gamma_q(t_s))\right), \quad t\geq t_s
\end{gather*}
with some $C_+>0$. It follows that $|{\bf z}(t)|\to 0$ as $t\to\infty$. In this case, the equilibrium $(0,0)$ is asymptotically stable.

Let $\beta_1^0>0$ and $\tilde\beta_2^0>0$, then integrating \eqref{vgeq} as $t\geq t_1$, we get
\begin{gather*}
\left|{\bf z}(t)\right|\geq C_- \exp\left( \frac{C_0}{2} (\gamma_n(t)-\gamma_n(t_1))\right) 
\end{gather*}
with some $C_->0$. In this case, the solution is unstable. Returning to the variables $(\rho,\phi)$, we obtain the stability or instability of the equilibrium $(\rho_0,\phi_0)$ in system \eqref{limsys}.
\end{proof}

\begin{proof}[Proof of Lemma~\ref{Lem2}]
Consider the case $q>n$. The proof in the case $q=n$ is similar.  Substituting \eqref{srp} into \eqref{trsys} and equating the terms of like powers of $\mu(t)$, we obtain
\begin{gather}\label{rpk}
{\bf Y}_{n,q}(\rho_0,\phi_0,1)
\begin{pmatrix} 
 \rho_k \\ \phi_k 
\end{pmatrix} 
= 
\begin{pmatrix} 
 F_k \\ G_k 
\end{pmatrix}, \quad k=1,\dots,N-q,
\end{gather}
where $F_k$, $G_k$ are expressed through $\rho_0,\phi_0, \dots, \rho_{k-1}, \phi_{k-1}$. For instance,
\begin{align*}
\begin{pmatrix} F_1\\ G_1 \end{pmatrix} = &- \begin{pmatrix}\Lambda_{n+1}(\rho_0,\phi_0) \\  \Omega_{q+1}(\rho_0,\phi_0) \end{pmatrix}, \\
\begin{pmatrix} F_2\\ G_2 \end{pmatrix} = 
	&- \begin{pmatrix}\Lambda_{n+2}(\rho_0,\phi_0) \\  \Omega_{q+2}(\rho_0,\phi_0) \end{pmatrix} 
	-\left(\rho_1\partial_\rho +\phi_1 \partial_\phi\right) \begin{pmatrix}\Lambda_{n+1}(\rho_0,\phi_0) \\  \Omega_{q+1}(\rho_0,\phi_0) \end{pmatrix} \\
	& -\frac{1}{2}\left(\rho_1^2\partial^2_\rho+2\rho_1\phi_1\partial_\rho\partial_\phi+\phi_1^2\partial^2_\phi\right)\begin{pmatrix}\Lambda_{n}(\rho_0,\phi_0) \\  \Omega_{q}(\rho_0,\phi_0) \end{pmatrix},\\
	\begin{pmatrix} F_3\\ G_3 \end{pmatrix} = 
	&- \begin{pmatrix}\Lambda_{n+3}(\rho_0,\phi_0) \\  \Omega_{q+3}(\rho_0,\phi_0) \end{pmatrix} 
	-\sum_{i+j=3}\left(\rho_i\partial_\rho +\phi_i \partial_\phi\right) \begin{pmatrix}\Lambda_{n+j}(\rho_0,\phi_0) \\  \Omega_{q+j}(\rho_0,\phi_0) \end{pmatrix} \\
	& -\frac{1}{6}\left(\rho_1^3\partial^3_\rho+3\rho_1^2\phi_1\partial_\rho^2\partial_\phi+3\rho_1\phi_1^2\partial_\rho\partial_\phi^2+\phi_1^3\partial^3_\phi\right)\begin{pmatrix}\Lambda_{n}(\rho_0,\phi_0) \\  \Omega_{q}(\rho_0,\phi_0) \end{pmatrix}.
\end{align*}
Since $\mathcal D_{n,q}\neq 0$, we see that system \eqref{rpk} is solvable. By construction, 
\begin{gather*}
		\begin{split}
&\mathcal R_\rho(t)\equiv \frac{d\rho_{\ast}(t)}{dt}-\Lambda(\rho_{\ast}(t),\phi_{\ast}(t),S(t),t)=\mathcal O\left(\mu^{N+1}(t)\right),\\
&\mathcal R_\phi(t)\equiv \frac{d\phi_{\ast}(t)}{dt}-\Omega(\rho_{\ast}(t),\phi_{\ast}(t),S(t),t)=\mathcal O\left(\mu^{N+1}(t)\right), \quad t\to\infty.
\end{split}
\end{gather*}
Consider the change of variables
\begin{gather} \label{subsM}
\rho(t)=\rho_{\ast}(t)+\mu^{\alpha}(t) z_1(t), \quad 
\phi(t)=\phi_{\ast}(t)+\mu^{\alpha+\frac{q-n}{2}}(t) z_2(t).
\end{gather}
Substituting \eqref{subsM} into \eqref{trsys}, we obtain  
\begin{gather}\label{z1z2sys}\begin{split}
\frac{dz_1}{dt}&={\mathcal F}(z_1,z_2,t) +\widetilde{\mathcal F}(z_1,z_2,t), \\ 
\frac{dz_2}{dt}&={\mathcal G}(z_1,z_2,t) +\widetilde{\mathcal G}(z_1,z_2,t)
\end{split}
\end{gather}
with  
\begin{gather*}
\begin{split}
{\mathcal F}(z_1,z_2,t)\equiv 
& \, 
\mu^{-\alpha}(t)
\left(
\widehat\Lambda\left(\rho_{\ast}(t)+ \mu^{\alpha}(t) z_1,\phi_{\ast}(t)+\mu^{\alpha+\frac{q-n}{2}}(t) z_2,t\right)
-\widehat\Lambda\left(\rho_\ast(t),\phi_\ast(t),t\right) 
\right) \\ 
& -\alpha  \ell(t) z_1, \\
{\mathcal G}(z_1,z_2,t)\equiv &\, \mu^{-\alpha-\frac{q-n}{2}}(t)\left(\widehat\Omega\left(\rho_{\ast}(t)+ \mu^{\alpha}(t) z_1,\phi_{\ast}(t)+\mu^{\alpha+\frac{q-n}{2}}(t) z_2,t\right)-\widehat\Omega\left(\rho_\ast(t),\phi_\ast(t),t\right) \right) \\
 & -\left(\alpha +\frac{q-n}{2}\right)\ell(t) z_2.\\
\widetilde{\mathcal F}(z_1,z_2,t)\equiv & \, 
\mu^{-\alpha}(t)
\left(
\widetilde\Lambda\left(\rho_{\ast}(t)+ \mu^{\alpha}(t) z_1,\phi_{\ast}(t)+\mu^{\alpha+\frac{q-n}{2}}(t) z_2, S(t),t\right)
-\widetilde\Lambda\left(\rho_\ast(t),\phi_\ast(t),S(t),t\right) 
\right) \\ 
& -\mu^{-\alpha}(t) \mathcal R_\varrho(t),\\
\widetilde{\mathcal G}(z_1,z_2,t)\equiv  & \mu^{-\alpha-\frac{q-n}{2}}(t)\left(\widetilde\Omega\left(\rho_{\ast}(t)+ \mu^{\alpha}(t) z_1,\phi_{\ast}(t)+\mu^{\alpha+\frac{q-n}{2}}(t) z_2, S(t),t\right)-\widetilde\Omega\left(\rho_\ast(t),\phi_\ast(t),S(t),t\right) \right) \\
& -\mu^{-\alpha-\frac{q-n}{2}}(t)\mathcal R_\phi(t).
\end{split}
\end{gather*}
Define
\begin{gather*}
\tilde \omega_{q,\alpha}:=\omega_q-\delta_{m,q}\chi_m \left(\alpha+\frac{ q-n }{2}\right), \quad \tilde \beta_{2,\alpha}^0:=\tilde \beta_2^0-\delta_{m,q}\chi_m \alpha.
\end{gather*}
It follows that 
\begin{align*}
{\mathcal F}(z_1,z_2,t) = &\,  
\left\{\mu^n(t)  \lambda_n z_1 \left(1 +\mathcal O(z_1 ) \right)+  \mu^{\frac{q+n}{2}}(t) \left(\xi_n z_2 + \mathcal O\left(|{\bf z}|^2\right)\right)\right\}\left(1+\mathcal O(\mu(t))\right), \\ 
{\mathcal G}(z_1,z_2,t) = &\, \left\{\mu^{\frac{q+n}{2}}(t)\eta_q z_1 \left( 1+\mathcal O(z_1)\right) +\mu^{q}(t)\left( \tilde\omega_{q,\alpha} z_2 \left(1+\mathcal O(\tilde\ell(t))+\mathcal O(\mu(t))\right)+ \mathcal O\left(|{\bf z}|^2\right)\right)\right\} \\
& \times\left(1+\mathcal O(\mu(t))\right),\\
\widetilde{\mathcal F}(z_1,z_2,t) = &\,  \mathcal O\left(\mu^{N+1-\alpha}(t)\right),\\
\widetilde{\mathcal G}(z_1,z_2,t) = &\,  \mathcal O\left(\mu^{N+1-\alpha-\frac{q-n}{2}}(t)\right)
\end{align*}
as $|{\bf z}|\to 0$ and $t\to\infty$.  Note that by choosing $\alpha \in(0,\alpha_0)$, we can ensure that $N+1-\alpha-(q-n)/2>q$ and $\tilde \beta_{2,\alpha}^0<0$. We use 
\begin{gather}\label{Valpha}
\mathcal V_\alpha(z_1,z_2,t)= B_{1,\alpha} z_1^2+ z_2^2+ \mu^{\frac{q-n}{2}}(t) B_{2,\alpha} z_1z_2 
\end{gather}
 as a Lyapunov function candidate for system \eqref{z1z2sys}, where the parameters $B_{1,\alpha}$ and $B_{2,\alpha}$ are chosen as
\begin{gather}\label{B1B2alpha}
B_{1,\alpha}=
\begin{cases}
\lambda_n \tilde \beta_{2,\alpha}, & \text{if} \quad \xi_n=0,\\
\displaystyle \frac{\lambda_n \tilde \beta_{2,\alpha}^0}{2\xi_n^2}, & \text{if} \quad \xi_n\neq 0,
\end{cases}
\quad  
B_{2,\alpha}=-\frac{2 B_{1,\alpha}\xi_n+2\eta_q}{\lambda_n}.
\end{gather}
Note that $\tilde \beta_{2,\alpha}=\tilde \omega_{q,\alpha}$ if $\xi_n=0$. Hence, $B_{1,\alpha}>0$. The derivative of $\mathcal V_\alpha(z_1,z_2,t)$ along the trajectories of system \eqref{z1z2sys} is given by 
\begin{gather*}
\frac{d\mathcal V_\alpha}{dt}={\mathfrak D}_\alpha(z_1,z_2,t)+\widetilde{\mathfrak D}_\alpha(z_1,z_2,t), 
\end{gather*}
where 
\begin{align*}
{\mathfrak D}_\alpha(z_1,z_2,t)\equiv & \, \left(\partial_t + {\mathcal F}(z_1,z_2,t)\partial_{z_1}+ {\mathcal G}(z_1,z_2,t)\partial_{z_2}\right)\mathcal V_\alpha(z_1,z_2,t), \\
\widetilde{\mathfrak D}_\alpha(z_1,z_2,t)\equiv & \, \left(\widetilde {\mathcal F}(z_1,z_2,t)\partial_{z_1}+\widetilde {\mathcal G}(z_1,z_2,t)\partial_{z_2}\right)\mathcal V_\alpha(z_1,z_2,t),
\end{align*}
It can easily be checked that
\begin{align*}
{\mathfrak D}_\alpha(z_1,z_2,t)= & \,  2B_{1,\alpha} \beta_1^0  \left( \mu^n(t) z_1^2 + \mu^q(t) (\xi_n^2+\delta_{\xi_n,0})\lambda_n^{-2}z_2^2 \right) \left (1+\mathcal O\left(|{\bf z}|\right)+\mathcal O\left(\mu(t)+|\tilde \ell(t)|\right) \right), \\
\widetilde{\mathfrak D}_\alpha(z_1,z_2,t)= & \, \mathcal O\left(|{\bf z}| \right)\mathcal O\left( \mu^{N+1-\alpha-\frac{q-n}{2}}(t)\right)
\end{align*}
as $|{\bf z}|\to 0$ and $t\to\infty$. Consequently, there exist $\Delta_1>0$ and $t_1\geq t_0$ such that
\begin{gather}
\label{Vineq} 
B_- |{\bf z}|^2\leq\mathcal V_\alpha (z_1,z_2,t)\leq B_+  |{\bf z}|^2,\\
\nonumber 
{\mathfrak D}_\alpha(z_1,z_2,t) \leq -\hat C_{\alpha}\mu^q(t)  |{\bf z}|^2, \quad
\widetilde{\mathfrak D}_\alpha(z_1,z_2,t) \leq \tilde C_{\alpha}\mu^q(t)  \varsigma(t) |{\bf z}|
\end{gather}
for all $|{\bf z}|\leq \Delta_1$ and $t\geq t_1$ with $B_+=2\max\{1,B_{1,\alpha}\}>0$, $B_-=\min\{1,B_{1,\alpha}\}/2>0$, $\hat C_{\alpha}=B_{1,\alpha}|\beta_1^0|\min\{1,(\xi_n^2+\delta_{\xi_n,0})\lambda_n^{-2}\}>0$ and some $\tilde C_{\alpha}>0$. Here, $\varsigma(t)\equiv \mu^{(2N+2+n-3q)/2-\alpha}(t)$ is a positive strictly decreasing function as $t\geq t_1$. 
This implies that for all $\epsilon\in (0,\Delta_1)$ there exist $\delta_\epsilon= 2\tilde C_{\alpha} \varsigma(t_\epsilon)/\hat C_{\alpha}$ and $t_\epsilon\geq t_1$ such that $\varsigma(t_\epsilon)\leq \epsilon \hat C_{\alpha}/ (2\tilde C_{\alpha})$ and the following inequality holds:
\begin{gather*}
\frac{d\mathcal V_\alpha}{dt} \leq 	\mu^q(t) \left(-\hat C_{\alpha}+\frac{\tilde C_{\alpha}\varsigma(t_\epsilon)}{\delta_\epsilon}\right)|{\bf z}|^2 \leq 0
\end{gather*}
for all $\delta_\epsilon\leq |{\bf z}|\leq \epsilon$ and $t\geq t_\epsilon$. Combining this with \eqref{Vineq}, we see that any solution of system \eqref{z1z2sys} with initial data $|{\bf z}(t_s)|\leq \delta$ as $t_s\geq t_\epsilon$, where $\delta=\min\{\delta_\epsilon,(\epsilon/2)\sqrt{B_-/B_+}\}$, cannot exit from the domain $|{\bf z}|\leq \epsilon$ as $t> t_s$. Hence, there exists a solution of system \eqref{trsys} such that $\rho_L(t)=\rho_{\ast}(t)+\mathcal O(\mu^{\alpha}(t))$, $\phi_L(t)=\phi_{\ast}(t)+\mathcal O(\mu^{\alpha+\frac{q-n}{2}}(t))$ as $t\to\infty$. 

To prove the stability of the solution $\rho_L(t)$, $\phi_L(t)$, consider the change of variables 
\begin{gather} \label{subsast}
\rho(t)=\rho_L(t)+z_1(t), \quad \phi(t)=\phi_L(t)+ \mu^{\frac{q-n}{2}}(t)z_2(t).
\end{gather} 
Substituting \eqref{subsast} into \eqref{trsys}, we obtain system \eqref{z1z2sys}, where
\begin{gather}\label{FG2}
\begin{split}
{\mathcal F}(z_1,z_2,t)\equiv 
& \, 
\Lambda\left(\rho_L(t)+z_1,\phi_L(t)+\mu^{\frac{q-n}{2}}(t) z_2, S(t),t\right)
-\Lambda\left(\rho_L(t),\phi_L(t),S(t),t\right), \\
{\mathcal G}(z_1,z_2,t)\equiv &\, \mu^{-\frac{q-n}{2}}(t)\left(\Omega\left(\rho_L(t)+z_1,\phi_L(t)+\mu^{\frac{q-n}{2}}(t) z_2, S(t),t\right)-\Omega\left(\rho_L(t),\phi_L(t),S(t),t\right) \right) \\
 & -\frac{q-n}{2}\ell(t) z_2,
\end{split}
\end{gather}
and $\widetilde{\mathcal F}(z_1,z_2,t)\equiv \widetilde{\mathcal G}(z_1,z_2,t)\equiv 0$. In this case,
\begin{gather}\label{FG2as}
\begin{split}
{\mathcal F}(z_1,z_2,t) = &\,  
\left\{\mu^n(t)  \lambda_n z_1 \left(1 +\mathcal O(z_1 ) \right)+  \mu^{\frac{q+n}{2}}(t) \left(\xi_n z_2 + \mathcal O\left(|{\bf z}|^2\right)\right)\right\}\left(1+\mathcal O(\mu(t))\right), \\ 
{\mathcal G}(z_1,z_2,t) = &\, \left\{\mu^{\frac{q+n}{2}}(t)\eta_q z_1 \left( 1+\mathcal O(z_1)\right) +\mu^{q}(t)\left( \tilde\omega_{q} z_2 \left(1+\mathcal O(\tilde\ell(t))+\mathcal O(\mu(t))\right)+ \mathcal O\left(|{\bf z}|^2\right)\right)\right\} \\
& \times\left(1+\mathcal O(\mu(t))\right)
\end{split}
\end{gather}
as $|{\bf z}|\to 0$ and $t\to\infty$, where $\tilde \omega_q=\omega_q-\delta_{m,q}(q-n)\chi_m/2$. By using Lyapunov function $\mathcal V_0(z_1,z_2,t)$  defined by \eqref{Valpha} and \eqref{B1B2alpha} with $\alpha=0$ and repeating the arguments as given above, we obtain
\begin{gather*}
\frac{d\mathcal V_0}{dt}\leq - \mu^q(t) C_0 \mathcal V_0\leq 0
\end{gather*}
 for all $|{\bf z}|\leq \Delta_2$, $t\geq t_2$ with some $0<\Delta_2>0$, $t_2> t_0$ and $C_0= \hat C_{0}/B_+>0$. Hence, the equilibrium $(0,0)$ of system \eqref{z1z2sys} is stable. Moreover, integrating the last inequality and using \eqref{Vineq}, we get
\begin{gather*}
\left|{\bf z}(t)\right|\leq C_+ \exp\left(-\frac{C_0}{2} (\gamma_q(t)-\gamma_q(t_2))\right), \quad t\geq t_2
\end{gather*}
with some $C_+>0$. It follows that if $\gamma_q(t)\to\infty$ as $t\to\infty$, then the equilibrium $(0,0)$ is asymptotically stable. Taking into account \eqref{subsast}, we obtain the stability of the solution $\rho_L(t)$, $\phi_L(t)$.
\end{proof}

\begin{proof}[Proof of Theorem~\ref{Th2}]
Substituting 
\begin{gather*}
r(t)=\rho_L(t)+z_1(t), \quad \psi(t)=\phi_L(t)+\mu^{\frac{q-n}{2}}(t) z_2(t)
\end{gather*}
 into \eqref{rpsi}, we obtain
\begin{gather}\label{zsde}
d {\bf z}(t)= \begin{pmatrix} \mathcal F(z_1,z_2,t) \\ 
\mathcal G(z_1,z_2,t)\end{pmatrix}dt+\varepsilon {\bf D}(z_1,z_2,t)\, d{\bf w}(t),
\end{gather}
where ${\bf D}(z_1,z_2,t)\equiv \{d_{i,j}(z_1,z_2,t)\}_{2\times 2}$, 
\begin{align*}
d_{1,j}(z_1,z_2,t) \equiv & \,  \sigma_{1,j}\left(\rho_L(t)+z_1,\phi_L(t)+ \mu^{\frac{q-n}{2}}(t) z_2,S(t),t\right),\\
d_{2,j}(z_1,z_2,t) \equiv & \, \mu^{-\frac{q-n}{2}}(t) \sigma_{2,j}\left(\rho_L(t)+z_1,\phi_L(t)+\mu^{\frac{q-n}{2}}(t)z_2,S(t)\right),
\end{align*}
the functions ${\mathcal F}(z_1,z_2,t)$, ${\mathcal G}(z_1,z_2,t)$ are defined by \eqref{FG2} and satisfy asymptotic estimates \eqref{FG2as}. 
It follows from \eqref{tildeLO} that
\begin{gather*}
d_{1,j}(z_1,z_2,t)=\mathcal O\left(\mu^p(t)\right), \quad
d_{2,j}(z_1,z_2,t)=\mathcal O\left(\mu^{\frac{2p-q+n}{2}}(t)\right)
\end{gather*}
as $t\to\infty$ uniformly for all $|{\bf z}|\leq \Delta_0$ with some $\Delta_0>0$. Note that $\mathcal F(0,0,t)\equiv \mathcal G(0,0,t)\equiv 0$. It follows from Lemma~\ref{Lem2} that system \eqref{zsde} with $\varepsilon=0$ has a stable equilibrium $(0,0)$. Let us prove the stability of the equilibrium under persistent perturbations by white noise by constructing a suitable stochastic Lyapunov function~\cite[Chapter II, \S 5]{HJK67}. 

Note that the generator of the process defined by \eqref{zsde} has the  form $\mathcal L\equiv  \mathcal L_0+\varepsilon^2 \mathcal L_1$, where
\begin{align*}
\mathcal L_0:=& \partial_t+\mathcal F \partial_{z_1}+ \mathcal G\partial_{z_2},\\
\mathcal L_1:=& \frac{1}{2}\left((d_{1,1}^2+d_{1,2}^2)\partial_{z_1}^2+2(d_{1,1}d_{2,1}+d_{1,2}d_{2,2})\partial_{z_1}\partial_{z_2}+(d_{2,1}^2+d_{2,2}^2)\partial_{z_2}^2\right).
\end{align*}
Consider the function $\mathcal U_0(z_1,z_2,t)\equiv \mathcal V_0(z_1,z_2,t)$, where $\mathcal V_0(z_1,z_2,t)$ is defined by \eqref{Valpha} and \eqref{B1B2alpha} with $\alpha=0$. It can easily be checked that 
\begin{align*}
 \mathcal L_0  \mathcal U_0(z_1,z_2,t) & = 2B_{1,0} \beta_1^0  \left( \mu^n(t) z_1^2 + \mu^q(t) (\xi_n^2+\delta_{\xi_n,0})\lambda_n^{-2}z_2^2 \right) \left (1+\mathcal O\left(|{\bf z}|\right)+\mathcal O\left(\mu(t)+|\tilde \ell(t)|\right) \right), \\
 \mathcal L_1 \mathcal U_0(z_1,z_2,t) & = \mathcal O\left(\mu^{2p-q+n}(t)\right)
\end{align*}
as $|{\bf z}|\to 0$ and $t\to\infty$, where $B_{1,0}>0$ is defined by \eqref{B1B2alpha}. Hence, there exist $0<\Delta_1\leq \Delta_0$ and $t_1\geq t_0$ such that
\begin{align}
\label{U0est}
 B_- |{\bf z}|^2\leq \mathcal U_0 (z_1,z_2,t)\leq B_+  |{\bf z}|^2,\\
\nonumber
\mathcal L \mathcal U_0(z_1,z_2,t) \leq -C_0 \mu^q(t)  |{\bf z}|^2+\varepsilon^2 C_1 \mu^{2p-q+n}(t),
\end{align}
for all $|{\bf z}|\leq \Delta_1$ and $t\geq t_1$ with $B_+=2\max\{1,B_{1,0}\}>0$, $B_-=\min\{1,B_{1,0}\}/2>0$, $C_{0}=B_{1,0}|\beta_1^0|\min\{1,(\xi_n^2+\delta_{\xi_n,0})\lambda_n^{-2}\}>0$ and some $C_1>0$. 
Fix the parameters $\epsilon_1\in (0,\Delta_1)$, $\epsilon_2\in (0,1)$ and $t_s\geq t_1$. Consider a stochastic Lyapunov function candidate
\begin{gather}\label{Vstdef}
\mathcal V(z_1,z_2,t)\equiv \mathcal U_0(z_1,z_2,t)+\varepsilon^2\mathcal U_1(t,\mathcal T),   
\end{gather}
where $\mathcal U_1(t,\mathcal T)\equiv 
C_1 \left(\gamma_{2p-q+n}(t_s+\mathcal T)-\gamma_{2p-q+n}(t)\right)$
with some parameter $\mathcal T>0$. It can easily be checked that 
\begin{gather}
\label{VLVest}
\begin{split}
\mathcal V(z_1,z_2,t) & \geq \mathcal U_0(z_1,z_2,t)\geq 0, \\ 
\mathfrak L \mathcal V(z_1,z_2,t) & \leq - C_1 \mu^q(t) |{\bf z}|^2\leq 0
\end{split}
\end{gather}
 for all $(z_1,z_2,t)\in \mathcal I(\Delta_1,t_\ast,\mathcal T):=\{(z_1,z_2,t)\in\mathbb R^3: |{\bf z}|\leq \Delta_1, 0\leq t-t_\ast\leq \mathcal T\}$. Let ${\bf z}(t)$ be a solution of system \eqref{zsde} with initial data $|{\bf z}(t_\ast)|\leq  \delta_1<\epsilon_1$ and $0<\varepsilon\leq \delta_2$. Positive parameters $\delta_1$, $\delta_2$ and $\mathcal T$ will be specified later. By $\theta_{\epsilon}$ denote the first exit time of $({\bf z}(t), t)$ from the domain $\mathcal I(\epsilon,t_\ast,\mathcal T)$. Define the function $\theta_{\epsilon_1}(t)\equiv \min\{\theta_{\epsilon_1}, t\}$. Then, ${\bf z}(\theta_{\epsilon_1}(t))$ is the process stopped at the first exit time from the domain $\mathcal I(\epsilon_1,t_\ast,\mathcal T)$. It follows from \eqref{VLVest} that $\mathcal V(z_1(\theta_{\epsilon_1}(t)), z_2(\theta_{\epsilon_1}(t)),\theta_{\epsilon_1}(t))$ is a nonnegative supermartingale (see, for example,~\cite[\S 5.2]{RH12}). Using Doob's inequality for supermartingales, we obtain 
\begin{align*}
\mathbb P\left(\sup_{0\leq t-t_\ast\leq \mathcal T} |{\bf z}(t)|\geq \epsilon_1\right) & = 
\mathbb P\left(\sup_{t\geq t_\ast} |{\bf z}(\theta_{\epsilon_1}(t))|\geq \epsilon_1\right)  \\
& \leq \mathbb P\left(\sup_{t\geq t_\ast} \mathcal V(z_1(\theta_{\epsilon_1}(t)), z_2(\theta_{\epsilon_1}(t)),\theta_{\epsilon_1}(t))\geq B_-\epsilon_1^2\right) \\
&\leq 
\frac{ \mathcal V\left (z_1(\theta_{\epsilon_1}(t_\ast)), z_2(\theta_{\epsilon_1}(t_\ast)\right), \theta_{\epsilon_1}(t_\ast))}{B_-\epsilon_1^2}
\\
&= 
\frac{ \mathcal V (z_1(t_\ast), z_2(t_\ast), t_\ast)}{B_-\epsilon_1^2}.
\end{align*}
It follows from \eqref{U0est} and \eqref{Vstdef} that 
\begin{gather*}
\mathcal V(z_1(t_\ast), z_2(t_\ast), t_\ast)\leq B_+ \delta_1^2+\varepsilon^{2(1-l)} \delta_2^{2l} \mathcal U_1(t_\ast,\mathcal T)
\end{gather*}
 with some $l\in(0,1]$. Consider first the case when $\mu^{2p-q+n}(t)\in L_1(t_0,\infty)$. It follows from \eqref{zetah} that there exists $\mathcal U_\ast>0$ such that $\mathcal U_1(t_\ast,\mathcal T)\leq \mathcal U_\ast$ for all $\mathcal T>0$. By taking $l=1$, 
\begin{align*}
 \delta_1=\epsilon_1\sqrt{ \frac{\epsilon_2 B_-}{2B_+}}, \quad 
 \delta_2=  \epsilon_1\sqrt{ \frac{\epsilon_2 B_-}{2\mathcal U_\ast}},
\quad 
\mathcal T=\infty,
\end{align*}
we obtain
\begin{gather*}
\mathbb P\left(\sup_{t\geq t_\ast} |{\bf z}(t)|\geq \epsilon_1\right) \leq \epsilon_2.
\end{gather*}
Consider now the case when $\mu^{2p-q+n}(t)\not\in L_1(t_0,\infty)$. Note that $\gamma_{2p-q+n}(t)\to\infty$ as $t\to\infty$ and $\mu(t)>0$ for all $t\geq \tau_0$. Hence, there exists $\mathcal T_\varepsilon>0$ such that 
\begin{gather*}
\mathcal U_1(t_\ast,\mathcal T_\varepsilon)=C_1 \varepsilon^{-2(1-l)} \quad \forall\, l \in (0,1).
\end{gather*}
Moreover, $\mathcal T_\varepsilon\to \infty$ as $\varepsilon \to 0$.
By taking
\begin{gather*}
 \delta_1=\epsilon_1\sqrt{ \frac{\epsilon_2 B_-}{2 B_+}}, \quad 
 \delta_2= \left( \frac{\epsilon_1^2\epsilon_2 B_-}{2 C_1}\right)^{\frac{1}{2l}},
\quad 
\mathcal T=\mathcal T_\varepsilon, 
\end{gather*}
we obtain
\begin{gather*}
\mathbb P\left(\sup_{0\leq t-t_\ast\leq \mathcal T_\varepsilon} |{\bf z}(t)|\geq \epsilon_1\right) \leq \epsilon_2.
\end{gather*}
Returning to the variables $r(t)$, $\psi(t)$, we get \eqref{rpst}.  
\end{proof}

\section{Analysis of the phase drift mode}
\label{sec5}

\begin{proof}[Proof of Lemma~\ref{Lem3}]
Substituting 
$
\rho(t)=\rho_0+\mu^{1/2}(t)z(t)
$ into \eqref{trsys}, we obtain
\begin{gather}\label{ypsi}
\frac{dz}{dt}= {\mathfrak F}_n(z,\phi,t)+\widetilde {\mathfrak F}_n(z,\phi,t), \quad 
\frac{d\phi}{dt}=\widetilde {\mathfrak G}_n(z,\phi,t),
\end{gather}
where 
\begin{align*}
{\mathfrak F}_n(z,\phi,t)\equiv & \, \Lambda_n\left(\rho_0+\mu^{\frac{1}{2}}(t)z,\phi\right)\mu^{n-\frac{1}{2}}(t)-\ell(t)\frac{z}{2},\\ 
\widetilde {\mathfrak F}_n(z,\phi,t)\equiv & \, \Lambda\left(\rho_0+\mu^{\frac{1}{2}}(t)z,\phi,S(t),t\right)\mu^{-\frac{1}{2}}(t)- \Lambda_n\left(\rho_0+\mu^{\frac{1}{2}}(t)z,\phi\right)\mu^{n-\frac{1}{2}}(t), \\
\widetilde {\mathfrak G}_n(z,\phi,t)\equiv & \, \Omega\left(\rho_0+\mu^{\frac{1}{2}}(t)z,\phi,S(t),t\right).
\end{align*} 
Note that 
\begin{gather}
\label{FFGas}
\begin{split}
	  {\mathfrak F}_n(z,\phi,t)= & \, \mu^{n}(t)z \left(\Xi_n(\phi) -\delta_{m,n}\frac{\chi_m}{2} +\mathcal O(z)+\mathcal O\left(\tilde\ell(t)\right)\right), \\
		\widetilde{\mathfrak F}_n(z,\phi,t)= & \, \mathcal O\left(\mu^{n+\frac{1}{2}}(t)\right),\\
		\widetilde {\mathfrak G}_n(z,\phi,t)= & \, \mu^q(t)\left(\Omega_q(\rho_0,\phi)+\mathcal O\left(\mu^{\frac{1}{2}}(t)\right)\right)
\end{split}
\end{gather}
as $z\to 0$ and $t\to\infty$ uniformly for all $\phi\in\mathbb R$.
Note that there exists $\xi_\ast>0$ such that $\Xi_n(\phi)-\delta_{m,n}\chi_m/2\leq -\xi_\ast$ for all $\phi\in\mathbb R$. From \eqref{FFGas} it follows that there exist $\Delta_1>0$ and $t_1\geq t_0$ such that 
\begin{gather}\label{phiineq}
\frac{d|z|}{dt}\leq \mu^n(t)\left(-\frac{\xi_\ast |z|}{2}+ \mu^{\frac{1}{2}}(t) C_\ast\right), \quad \left|\frac{d\phi}{dt}\right|\geq \mu^q(t) \mathfrak G_\ast
\end{gather}
for all $|z|\leq \Delta_1$, $\phi\in\mathbb R$ and $t\geq t_1$ with some constants $C_\ast>0$ and $\mathfrak G_\ast>0$.
Therefore, for all $\epsilon\in (0,\Delta_1)$ there exist $\delta_\epsilon= 4 C_\ast \sqrt{\mu(t_\epsilon)}/|\xi_\ast|$ and $t_\epsilon\geq t_1$ such that $\sqrt{\mu(t_\epsilon)}\leq \epsilon |\xi_\ast|/ (4 C_\ast)$ and    
\begin{gather*}
\frac{d|z|}{dt}\leq \mu^n(t)|z|\left(-\frac{\xi_\ast}{2}+\mu^{\frac{1}{2}}(t_\epsilon)\frac{C_\ast}{\delta_\epsilon}\right)\leq -\mu^n(t)\frac{\xi_\ast|z|}{4}
\end{gather*}
for all $\delta_\epsilon\leq |z|\leq \epsilon$, $\phi\in\mathbb R$ and $t\geq t_\epsilon$. Hence, any solution $(z(t),\phi(t))$ of system \eqref{ypsi} with initial data $|z(t_s)|\leq \delta_\epsilon$, $\phi(t_s)\in\mathbb R$, $t_s\geq t_\epsilon$ satisfies the inequality $|z(t)|<\epsilon$ for all $t\geq t_s$. Thus, there exists a solution $\tilde \rho_D(t)$, $\tilde \phi_D(t)$ of system \eqref{trsys} such that $\tilde \rho_D(t)=\rho_0+\mathcal O(\mu^{1/2}(t))$ as $t\to\infty$. Integrating the second inequality in \eqref{phiineq}, we get $\tilde \phi_D(t)\to \infty$ as $t\to\infty$.

\end{proof}

\begin{proof}[Proof of Lemma~\ref{Lem4}]
Consider the transformation
\begin{gather}\label{zch}
	E_N(r,\psi,t)=r+\sum_{k=n-q}^{N-q}  e_k(r,\psi) \mu^{k}(t), 
\end{gather}
with $N\in [n,m]$. The functions $e_k(r,\psi)$ are assumed to be periodic with respect in $\psi$ and are sought in such a way that the system written in the new variables 
$z(t)\equiv E_N(r(t),\psi(t),t)$, $\psi(t)$
takes the form \eqref{rpsi2}. 

Denote by $\mathfrak L$ the generator of the process defined by system \eqref{rpsi}. Then,
\begin{gather*}
\mathfrak L U:= 	 (\partial_t + \Lambda \partial_r+\Omega \partial_\psi) U+\frac{\varepsilon^2}{2}{\hbox{\rm tr}}\left({\bf C}^T {\bf H}_{(r,\psi)}(U){\bf C}\right) 
\end{gather*}
for any smooth function $U(r,\psi,t)$, where ${\bf H}_{(r,\psi)}(U)$ is defined by \eqref{HJdef}. By applying It\^{o}'s formula, we get
\begin{gather}\label{Itoz}
d z =\mathfrak L  E_N( r,\psi,t)\, dt + \varepsilon (\nabla_{(r,\psi)} E_N(r,\psi,t))^T {\bf C}(r,\psi,S(t),t) \, d{\bf w}(t).
\end{gather}
It follows from \eqref{FGas}, \eqref{tildeas} and \eqref{Cform} that
\begin{gather*}
{\bf C}(r,\psi,S(t),t)= \sum_{k=p}^{p+N} {\bf C}_k(r,\psi,S(t))\mu^{k}(t)+\tilde {\bf C}_{p+N}(r,\psi,t), \quad 
\|\tilde {\bf C}_{p+N}(r,\psi,t)\|=\mathcal O\left(\mu^{p+N+1}(t)\right)
 \end{gather*}
as  $t\to\infty$ uniformly for all $r\in(0,\mathcal R]$, $(\psi,S)\in\mathbb R^2$. Hence, the drift term in \eqref{Itoz} takes the following form:
\begin{gather}\label{dz}
\begin{split}
\mathfrak L E_N  = &
\sum_{k=n}^{N} \mu^k(t)\left\{\Lambda_k-\frac{\kappa s_q}{\varkappa}\partial_\psi e_{k-q}  \right\} + 
\sum_{k=2n-q}^N \mu^k(t) \sum_{\substack{i+j=k\\ n\leq i\leq N \\ n-q\leq j\leq N-q}} \Lambda_{i}\partial_r e_{j}+ 
\sum_{k=n+1}^N \mu^k(t) \sum_{\substack{i+j=k\\ q+1\leq i\leq N \\ n-q\leq j\leq N-q}} \Omega_{i}\partial_\psi e_{j}\\
 & +\frac{\varepsilon^2}{2}\sum_{k=n+2p-q}^N \mu^k(t) \sum_{\substack{i+j+l=k\\ p\leq i,j\leq p+N \\ n-q\leq l\leq N-q}} 
{\hbox{\rm tr}}\left({\bf C}^T_{i} {\bf H}_{(r,\psi)}(e_{l} ){\bf C}_{j}\right)+ \tilde{\mathfrak R}_N(r,\psi,t),
\end{split}
\end{gather}
where  
\begin{align*}
\tilde{\mathfrak R}_N(r,\psi,t)\equiv & \, \widetilde\Lambda_N(r,\psi,S(t),t)+
\sum_{k=N+1}^{2N-q} \mu^k(t) \sum_{\substack{i+j=k\\ n\leq i\leq N \\ n-q\leq j\leq N-q}} \Lambda_{i}\partial_r e_{j}+ 
\sum_{k=N+1}^{2N-q} \mu^k(t) \sum_{\substack{i+j=k\\ q+1\leq i\leq N \\ n-q\leq j\leq N-q}} \Omega_{i}\partial_\psi e_{j}\\
 & +\frac{\varepsilon^2}{2}\sum_{k=N+1}^{2p+3N-q} \mu^k(t) \sum_{\substack{i+j+l=k\\ p\leq i,j\leq p+N \\ n-q\leq l\leq N-q}} {\hbox{\rm tr}}\left ({\bf C}^T_{i} {\bf H}_{(r,\psi)}(e_{l} ){\bf C}_{j}\right)\\
 & +\sum_{k= n-q}^{N-q} \mu^k(t) \left\{ \left(k\ell(t)+\tilde \Lambda_{N}\partial_r+\tilde \Omega_N\partial_\psi\right)e_k+\frac{\varepsilon^2}{2}{\hbox{\rm tr}}\left (\tilde {\bf C}^T_{p+N} {\bf H}_{(r,\psi)}(e_{k} )\tilde{\bf C}_{p+N}\right) \right\}.
\end{align*}
It follows easily that 
\begin{gather}\label{tildeRN}
|\tilde {\mathfrak R}_{N}(r,\psi,t)|=\mathcal O\left(\mu^{N+1}(t)\right)
\end{gather}
 as $t\to\infty$ uniformly for all $r\in (0,\mathcal R]$ and $\psi\in\mathbb R$. Comparing the drift term of the first equation in \eqref{rpsi2} with \eqref{dz}, we obtain
\begin{gather}\label{ek}
-\frac{\kappa s_q}{\varkappa}\partial_\psi e_{k-q}=\mathcal F_k(r)-\Lambda_k(r,\psi)+\tilde g_k(r,\psi), \quad k=n,\dots, N,
\end{gather} 
where each $\tilde g_k(r,\psi)$ is expressed in terms of $\{e_{j-q}, \mathcal F_{j}\}_{j=n}^{k-1}$. In particular, if $p=q=1$, $n=2$, then
$\tilde g_2 \equiv 0$, $\tilde g_3 \equiv e_1\partial_r \mathcal F_2
-(\Lambda_2\partial_r + \Omega_2\partial_\psi )e_1- \varepsilon^2 
{\hbox{\rm tr}} ({\bf C}^T_{1} {\bf H}_{(r,\psi)}(e_{1} ){\bf C}_{1})/2$.
Define 
\begin{gather*}
\mathcal F_k(r) \equiv \left\langle \Lambda_k(r,\psi)-\tilde g_k(r,\psi)\right\rangle_{\psi}.
\end{gather*}
Then, system \eqref{ek} is solvable in the class of functions that are $2\pi$-periodic in $\psi$ with zero mean.
From \eqref{zch} it follows that for all $\epsilon\in (0,\mathcal R/2)$ there exists $t_1\geq t_0$ such that $|E_N(r,\psi,t)-r|< \epsilon$, $|\partial_r E_N(r,\psi,t)-1|<\epsilon$ for all $r\in(0,\mathcal R]$, $\psi\in\mathbb R$ and $t\geq t_1$. Hence, there exists the inverse transformation $r= \mathfrak R(z,\psi,t)$ such that $0<\mathfrak R(z,\psi,t)<\mathcal R$ for all $z\in [\epsilon,\mathcal R-\epsilon]$, $\psi\in\mathbb R$ and $t\geq t_1$.
Define
\begin{align*}
\tilde {\mathcal F}_N(z,\psi,t)
\equiv & \,
\mathfrak L E_N \left(\mathfrak R(z,\psi,t),\psi,t\right) - \sum_{k=n}^N \mathcal F_k(z) \mu^k(t), \\
\mathcal G(z,\psi,t) \equiv & \, \Omega(\mathfrak R(z,\psi,t),\psi,S(t),t), \\
{\bf Z}(z,\psi,t)\equiv & \,
{\bf J}_{(r,\psi)} \left(E_N(\mathfrak R(z,\psi,t),\psi,t),\psi\right) {\bf C}(\mathfrak R(z,\psi,t),\psi,S(t),t),
\end{align*}
where ${\bf J}_{(r,\psi)}(U,V)$ is defined by \eqref{HJdef}.
Combining this with \eqref{tildeLO}, \eqref{zch} and \eqref{tildeRN}, we get \eqref{tildeLO2}. Thus, we obtain the proof of Lemma~\ref{Lem4} with $\tilde E_N(r,\psi,t)\equiv  E_N(r,\psi,t)-r$.

\end{proof}

\begin{proof}[Proof of Lemma~\ref{Lem5}]
Consider the function
\begin{gather}\label{rhostar}
\rho_\star(t)=\rho_0+\sum_{k=1}^{N-n}\rho_k \mu^k(t), 
\end{gather}
where $\{\rho_k\}_{k=1}^{N-n}$ are some constants. 
Substituting \eqref{rhostar} in the first equation in \eqref{trsys2} and grouping the terms of like powers of $\mu(t)$, we obtain a chain of equations for $\rho_k$: $\hat\xi_n \rho_k=F_k$, where $F_k$ is expressed through $\rho_0,\rho_1, \dots, \rho_{k-1}$. In particular, $F_1=-\mathcal F_{n+1}(\rho_0)$, $F_2=-\mathcal F_{n+2}(\rho_0)-\rho_1 \mathcal F_{n+1}'(\rho_0)-\rho_1^2\mathcal F_{n}''(\rho_0)/2 $. Since $\hat\xi_{n}\neq 0$, we see that this system is solvable. Moreover, by construction, 
\begin{gather*}
\mathcal R_\star(t)\equiv \frac{d\rho_{\star}(t)}{dt}-\widehat{\mathcal F}(\rho_\star(t),t)=\mathcal O\left(\mu^{N+1}(t)\right), \quad t\to\infty.
\end{gather*}
Substituting 
$
\rho(t)=\rho_\star(t)+u(t) \mu^{1/2}(t)
$ into \eqref{trsys2}, we obtain
\begin{gather}\label{ypsi2}
\frac{du}{dt}= \widehat {\mathfrak F}_\star(u,t)+\widetilde {\mathfrak F}_\star(u,\phi,t), \quad 
\frac{d\phi}{dt}=-\frac{\kappa s_q}{\varkappa}\mu^q(t)+\widetilde {\mathfrak G}_\star(u,\phi,t),
\end{gather}
where 
\begin{align*}
\widehat {\mathfrak F}_\star(u,t)\equiv & \, \left(\widehat{\mathcal F}\left(\rho_\star(t)+u\mu^{\frac{1}{2}}(t),t\right)-\widehat{\mathcal F}(\rho_\star(t),t)\right)\mu^{-\frac{1}{2}}(t)-\ell(t)\frac{u}{2},\\ 
\widetilde {\mathfrak F}_\star(u,\phi,t)\equiv & \, \left(\widetilde{\mathcal F}\left(\rho_\star(t)+u\mu^{\frac{1}{2}}(t),\phi,t\right)-\mathcal R_\star(t)\right)\mu^{-\frac{1}{2}}(t), \\
\widetilde {\mathfrak G}_\star(u,\phi,t)\equiv & \, \widetilde{\mathcal G}_q\left(\rho_\star(t)+u\mu^{\frac{1}{2}}(t),\phi,t\right).
\end{align*} 
It can easily be checked that 
\begin{gather}
\label{FFGas2}
\begin{split}
	  \widehat {\mathfrak F}_\star(u,t)= & \, \mu^{n}(t)u \left(\hat\xi_n -\delta_{m,n}\frac{\chi_m}{2}+\mathcal O(u)+\mathcal O\left(\varsigma(t)\right)\right), \\
		\widetilde{\mathfrak F}_\star(u,\phi,t)= & \, \mathcal O\left(\mu^{N+\frac{1}{2}}(t)\right),\\
		\widetilde {\mathfrak G}_\star(u,\phi,t)= & \, \mathcal O\left(\mu^{q+1}(t)\right)
\end{split}
\end{gather}
as $u\to 0$ and $t\to\infty$ uniformly for all $\phi\in\mathbb R$, where $\varsigma(t)\equiv \mu^{1/2}(t)+|\tilde\ell(t)|$ is strictly decreasing.
Define $\hat\xi_\star=\hat\xi_n-\delta_{m,n}\chi_m/2<0$. From \eqref{FFGas2} it follows that there exist $\Delta_1>0$ and $t_2\geq t_1$ such that 
\begin{gather}\label{phiineq2}
\frac{d|u|}{dt}\leq \mu^n(t)\left(-\frac{|\hat\xi_\star| |u|}{2}+ \varsigma(t) C_1\right), \quad 
	\left|\frac{d\phi}{dt}+\frac{\kappa s_q}{\varkappa}\mu^q(t)\right|\leq C_2\mu^{q+1}(t) 
\end{gather} 
for all $|u|\leq \Delta_1$, $\phi\in\mathbb R$ and $t\geq t_2$ with some $C_1>0$ and $C_2>0$.
Therefore, for all $\epsilon\in (0,\Delta_1)$ there exist $\delta_\epsilon= 4 C_1 \varsigma(t_\epsilon)/|\hat\xi_\star|$ and $t_\epsilon\geq t_2$ such that $\varsigma(t_\epsilon)\leq \epsilon |\hat\xi_\star|/ (4 C_1)$ and    
\begin{gather*}
\frac{d|u|}{dt}\leq \mu^n(t)|u|\left(-\frac{|\hat\xi_\star|}{2}+\varsigma(t_\epsilon)\frac{C_1}{\delta_\epsilon}\right)\leq -\mu^n(t)\frac{\hat\xi_\star|u|}{4}
\end{gather*}
for all $\delta_\epsilon\leq |u|\leq \epsilon$, $\phi\in\mathbb R$ and $t\geq t_\epsilon$. Hence, any solution $(u(t),\phi(t))$ of system \eqref{ypsi2} with initial data $|u(t_s)|\leq \delta_\epsilon$, $\phi(t_s)\in\mathbb R$, $t_s\geq t_\epsilon$ satisfies the inequality $|u(t)|<\epsilon$ for all $t\geq t_s$. Thus, there exists a solution $\rho_D(t)$, $\phi_D(t)$ of system \eqref{trsys2} such that $\rho_D(t)=\rho_0+\mathcal O(\mu^{1/2}(t))$ as $t\to\infty$. Integrating the second inequality in \eqref{phiineq2}, we obtain $\phi_D(t)=-\kappa s_q \varkappa^{-1}\gamma_q(t) (1+o(1))$ as $t\to\infty$.
\end{proof}

\begin{proof}[Proof of Theorem~\ref{Th3}]
Let $\rho_D(t)$, $\phi_D(t)$ be the solution of the system \eqref{trsys2} described in Lemma~\ref{Lem5}. Substituting 
\begin{gather}\label{rhodz}
z(t)=\rho_D(t)+u(t)
\end{gather}
into \eqref{rpsi2}, we obtain
\begin{gather}\label{upe}
d \begin{pmatrix}u \\ 
\psi\end{pmatrix}= \begin{pmatrix} \mathfrak F_D(u,\psi,t) \\ 
\mathfrak G_D(u,\psi,t)\end{pmatrix}dt+\varepsilon {\bf D}(u,\psi,t)\, d{\bf w}(t),
\end{gather}
where ${\mathfrak F}_D(u,\psi,t)\equiv \widehat{\mathfrak F}_D(u,t)+\widetilde{\mathfrak F}_D(u,\psi,t)$, ${\bf D}(u,\psi,t)\equiv \{d_{i,j}(z,\phi,t)\}_{2\times 2}$,
\begin{gather*}
\begin{split}
\widehat{\mathfrak F}_D(u,t)\equiv 
& \, 
\widehat{\mathcal F}_N\left(\rho_D(t)+u,t\right)
-\widehat{\mathcal F}_N\left(\rho_D(t),t\right), \\
\widetilde{\mathfrak F}_D(u,\psi,t)\equiv 
& \, 
\widetilde{\mathcal F}_N\left(\rho_D(t)+u,\psi,t\right)
-\widetilde{\mathcal F}_N\left(\rho_D(t),\phi_D(t),t\right), \\
{\mathfrak G}_D(u,\psi,t)\equiv &\, \mathcal G\left(\rho_D(t)+u,\psi,t\right), \\
d_{i,j}(u,\psi,t) \equiv & \,  \zeta_{i,j}\left(\rho_D(t)+u,\psi,t\right).
\end{split}
\end{gather*}
Note that $\widehat {\mathfrak F}_D(0,t)\equiv 0$ and  $\widehat {\mathfrak F}_D(u,t) = \mu^n(t)u \left(\hat\xi_n+\mathcal O(u)+\mathcal O\left(\mu^{1/2}(t)\right)\right)$ as $u\to 0$, $t\to\infty$ uniformly for all $\psi\in\mathbb R$.
It follows from \eqref{tildeLO2} that
\begin{align*}
\widetilde {\mathfrak F}_D(u,\psi,t) =  \mathcal O\left(\mu^{N+1}(t)\right), \quad 
\mathfrak G_D(u,\psi,t) =  \mu^q(t)\left(-\frac{\kappa s_q}{\varkappa} +\mathcal O\left(\mu(t)\right)\right),\quad
d_{i,j}(u,\psi,t)=  \mathcal O\left(\mu^{p}(t)\right)
\end{align*}
as $t\to\infty$ uniformly for all $|u|\leq \Delta_1$ and $\psi\in\mathbb R$ with some $\Delta_1>0$. Denote by $\mathfrak L\equiv  \widehat{\mathfrak L}_0+\widetilde {\mathfrak L}_0+\varepsilon^2 \mathfrak L_1$ the generator of the process defined by \eqref{upe}, where
\begin{align*}
\widehat{\mathfrak L}_0:=& \partial_t+\widehat {\mathfrak F}_D \partial_{u}+ \mathfrak G_D\partial_{\psi},\\
\widetilde {\mathfrak L}_0:=&\widetilde {\mathfrak F}_D \partial_{u},\\
\mathfrak L_1:=& \frac{1}{2}\left((d_{1,1}^2+d_{1,2}^2)\partial_{u}^2+2(d_{1,1}d_{2,1}+d_{1,2}d_{2,2})\partial_{u}\partial_{\psi}+(d_{2,1}^2+d_{2,2}^2)\partial_{\psi}^2\right).
\end{align*}
Consider an auxiliary function $\mathfrak U_0(u)\equiv u^2/2$. It can easily be checked that 
\begin{align*}
\widehat{\mathfrak L}_0  \mathfrak U_0(u) & = \mu^n(t)u^2 \left(\hat\xi_n+\mathcal O(u)+\mathcal O\left(\mu^{\frac{1}{2}}(t)\right)\right), \\
\widetilde{\mathfrak L}_0  \mathfrak U_0(u) & = \mathcal O(u) \mathcal O\left(\mu^{N+1}(t)\right), \\
 \mathfrak L_1 \mathfrak U_0(u)& = \mathcal O\left(\mu^{2p}(t)\right)
\end{align*}
as $u\to 0$ and $t\to\infty$ uniformly for all $\psi\in\mathbb R$. Hence, there exist $0<\Delta_2\leq \Delta_1$ and $t_2\geq t_1$ such that
\begin{gather}
\begin{split}
\label{Vineq2} 
\widehat{\mathfrak L}_0 \mathfrak U_0(u) \leq & - \mu^n(t) \frac{3|\hat\xi_n|u^2}{4} , \\
\widetilde{\mathfrak L}_0 \mathfrak U_0(u) \leq &  2 \tilde C_0 \mu^{N+1}(t)|u| \leq \tilde C_0 \mu^{N+\frac{1}{2}}(t)\left(\frac{u^2}{\varepsilon^2}+\varepsilon^2 \mu(t)\right)\leq \mu^n(t)\frac{|\hat\xi_n| u^2}{4}+ \tilde C_0\varepsilon^2\mu^{n+1}(t),\\
{\mathfrak L}_1 \mathfrak U_0(u) \leq & C_1 \mu^{2p}(t)
\end{split}
\end{gather}
for all $|u|\leq \Delta_2$, $t\geq t_2$ and $\psi\in\mathbb R$ with some $\tilde C_0>0$ and $C_1>0$. 
Fix the parameters $\epsilon_1\in (0,\Delta_2)$, $\epsilon_2\in (0,1)$, $t_s\geq t_2$, and consider a stochastic Lyapunov function candidate in the following form:
\begin{gather*}
\mathfrak V(u,t)\equiv \mathfrak U_0(u)+\varepsilon^2\mathfrak U_1(t,\mathcal T),   
\end{gather*}
where $\mathfrak U_1(t,\mathcal T)\equiv 
C_2 \left(\gamma_{c}(t_s+\mathcal T)-\gamma_{c}(t)\right)$
with $C_2=\tilde C_0+C_1>0$ and some parameter $\mathcal T>0$. It follows from \eqref{Vineq2} that 
\begin{gather}
\label{LUest}
\begin{split}
\mathfrak V(u,t)  \geq \frac{u^2}{2}\geq 0, \quad 
\mathfrak L \mathfrak V(u,t)  \leq -  \mu^{n}(t) \frac{|\hat\xi_n|u^2}{2}\leq 0
\end{split}
\end{gather}
 for all $(u,\psi,t)\in \mathfrak I(\Delta_2,t_s,\mathcal T):=\{(u,\psi,t)\in\mathbb R^3: |u|\leq \Delta_2, 0\leq t-t_s\leq \mathcal T\}$. Let $(u(t),\psi(t))$ be a solution of system \eqref{upe} with initial data $|u(t_s)|\leq  \delta_1<\epsilon_1$, $\psi(t_s)\in\mathbb R$ and $0<\varepsilon\leq \delta_2$. Positive parameters $\delta_1$, $\delta_2$ and $\mathcal T$ will be specified later. By $\theta_{\epsilon_1}$ denote the first exit time of $(u(t),\psi(t), t)$ from the domain $\mathfrak I(\epsilon,t_s,\mathcal T)$. Define the function $\theta_{\epsilon_1}(t)\equiv \min\{\theta_{\epsilon_1}, t\}$. Then, $(u(\theta_{\epsilon_1}(t)),\psi(\theta_{\epsilon_1}(t)),\theta_{\epsilon_1}(t))$ is the process stopped at the first exit time from the domain $\mathfrak I(\epsilon_1,t_s,\mathcal T)$. It follows from \eqref{LUest} that $\mathfrak V(u(\theta_{\epsilon_1}(t)),\theta_{\epsilon_1}(t))$ is a nonnegative supermartingale. By applying Doob's inequality, we obtain 
\begin{align*}
\mathbb P\left(\sup_{0\leq t-t_s\leq \mathcal T} |u(t)|\geq \epsilon_1\right) & = 
\mathbb P\left(\sup_{t\geq t_s} |u(\theta_{\epsilon_1}(t))|\geq \epsilon_1\right)  \\
& \leq \mathbb P\left(\sup_{t\geq t_s} \mathfrak V(u(\theta_{\epsilon_1}(t)),\theta_{\epsilon_1}(t))\geq \frac{\epsilon_1^2}{2}\right) \\
&\leq 
\frac{ 2\mathfrak V\left (u(\theta_{\epsilon_1}(t_s)), \theta_{\epsilon_1}(t_s)\right)}{\epsilon_1^2}
\\
&= 
\frac{ 2\mathfrak V (u(t_s),t_s)}{\epsilon_1^2}.
\end{align*}
It can easily be checked that 
$\mathfrak V(u(t_s), t_s)\leq {\delta_1^2}/{2}+\varepsilon^{2(1-l)} \delta_2^{2l} \mathfrak U_1(t_s,\mathcal T)$ 
 with some $l\in(0,1]$. Let $\mu^{c}(t)\in L_1(t_0,\infty)$. Then, there exists $\mathfrak U_\ast>0$ such that $\mathfrak U_1(t_s,\mathcal T)\leq \mathfrak U_\ast$ for all $\mathcal T>0$. By taking 
\begin{align*}
l=1, \quad 
\delta_1=\epsilon_1\sqrt{ \frac{\epsilon_2}{2}}, \quad 
\delta_2=  \epsilon_1\sqrt{ \frac{\epsilon_2 }{4\mathfrak U_\ast}},
\quad 
\mathcal T=\infty,
\end{align*}
we obtain
\begin{gather*}
\mathbb P\left(\sup_{t\geq t_s} |u(t)|\geq \epsilon_1\right) \leq \epsilon_2.
\end{gather*}
Now, let $\mu^{c}(t)\not\in L_1(t_0,\infty)$. In this case, $\gamma_{c}(t)\to\infty$ as $t\to\infty$. Hence, there exists $\mathcal T_\varepsilon>0$ such that $\mathfrak U_1(t_\ast,\mathcal T_\varepsilon)=C_2 \varepsilon^{-2(1-l)}$ for all  $l \in (0,1)$. It is clear that $\mathcal T_\varepsilon\to \infty$ as $\varepsilon \to 0$.
By taking
\begin{gather*}
 \delta_1=\epsilon_1\sqrt{ \frac{\epsilon_2}{2}}, \quad 
 \delta_2= \left( \frac{\epsilon_1^2\epsilon_2 }{4 C_2}\right)^{\frac{1}{2l}},
\quad 
\mathcal T=\mathcal T_\varepsilon, 
\end{gather*}
we obtain
\begin{gather*}
\mathbb P\left(\sup_{0\leq t-t_s\leq \mathcal T_\varepsilon} |u(t)|\geq \epsilon_1\right) \leq \epsilon_2.
\end{gather*}
Combining this with \eqref{ch3} and \eqref{rhodz}, we get \eqref{rpD}. 
\end{proof}

\section{Examples}\label{sex}

\subsection{Example 1.} Consider again system \eqref{Ex0}. It was shown in Section~\ref{sec1} that this system corresponds to \eqref{PS} with $\nu_0=1$, $\mu(t)\equiv t^{-1/4}$, and ${\bf a}(\varrho,\varphi,S,t)=(a_1(\varrho,\varphi,S,t),a_2(\varrho,\varphi,S,t))^T$, ${\bf A}(\varrho,\varphi,S,t)=\{\alpha_{i,j}(\varrho,\varphi,S,t)\}_{2\times 2}$ defined by \eqref{fgex0str} and \eqref{fgex0}. It can easily be checked that the condition \eqref{mucond} holds with $m=4$ and $\chi_m=-1/4$. Let $s_0=2$, then there exist $\kappa=1$, $\varkappa=2$ such that the resonance condition \eqref{rc} holds. 

Let $n=2$ and $p=1$. Then, the change of variables described in Theorem~\ref{Th1} with $N=2$ transforms the system into \eqref{rpsi} with 
\begin{align*}
\Lambda(r,\psi,S,t) \equiv &\,  t^{-\frac{1}{2}}\Lambda_2(r,\psi)+\widetilde \Lambda_1(r,\psi,S,t), \\ 
\Omega(r,\psi,S,t)\equiv &\, t^{-\frac{1}{4}}\Omega_1(r,\psi) + t^{-\frac{1}{2}}\Omega_2(r,\psi)+\widetilde \Omega_2(r,\psi,S,t),
\end{align*}
where
\begin{align*}
\Lambda_2(r,\psi)\equiv & \, \frac{r}{64} \left(32 B_0 + 24  C_0 r^2 + 16  Q_1 \sin 2 \psi + \varepsilon^2 (6 - \cos 4 \psi) \right), \\ 
\Omega_1(r,\psi)\equiv & \, -\frac{s_1}{2}, \\
\Omega_2(r,\psi)\equiv & \, \frac{\cos 2 \psi }{16}  (4  Q_1 +  \varepsilon^2 \sin 2 \psi),
\end{align*}
and $\widetilde \Lambda_2(r,\psi,S,t)=\mathcal O(t^{-3/4})$, $\widetilde \Omega_2(r,\psi,S,t)=\mathcal O(t^{-3/4})$ as $t\to\infty$ uniformly for all $r\in (0,\mathcal R]$, $(\psi,S)\in\mathbb R^2$ with any fixed $\mathcal R>0$; $ Q_1= B_1- A_1$. 

Let $s_1=0$. Then, assumption \eqref{asq} holds with $q=n=2$. Note that the corresponding limiting system \eqref{limsys} may have several fixed points. 

If $(32 B_0\pm 16 Q_1 +7\varepsilon^2)  C_0<0$ and $\varepsilon^2\pm 4Q_1\neq 0$, then there exist
\begin{gather*}
\rho_0^\pm=\sqrt{-\frac{32 B_0\pm 16 Q_1 +7\varepsilon^2}{24  C_0}}, \quad \phi_0^\pm= \pm\frac{\pi}{4} ({\hbox{\rm mod }} \pi)
\end{gather*}
such that assumption \eqref{as1} holds with $\mathcal D_{n,q}^\pm=-3  C_0 (\rho_0^\pm)^2 (\varepsilon^2\pm 4 Q_1 )/32\neq 0$. It can easily be checked that 
\begin{gather*}
\lambda_n^\pm=\frac{3 C_0(\rho_0^\pm)^2}{4}, \quad 
\xi_n^\pm=\eta_q^\pm=0, \quad 
\omega_q^\pm=-\frac{\varepsilon^2 \pm 4 Q_1 }{8}.
\end{gather*}
In this case, $\beta_1^0=\lambda_n^\pm$ and $\tilde\beta_2^0=\beta_2^0=\omega_q^\pm$. If, in addition, $ C_0<0$ and $\varepsilon^2\pm 4 Q_1>0$, then it follows from Lemma~\ref{Lem2} that there is a stable phase locking solution $\rho_L^\pm(t)$, $\phi_L^\pm(t)$ of the corresponding truncated system \eqref{trsys} such that $\rho_L^\pm(t)\sim \rho_0^\pm$, $\phi_L^\pm(t)\sim \phi_0^\pm$ as $t\to\infty$. Moreover, from Theorem~\ref{Th2} it follows that for all $l\in(0,1)$ this regime is stochastically stable in system \eqref{Ex0} on an asymptotically long time interval as  $t_s\leq t\leq    \varepsilon^{-4(1-l)}/4$. Note that if $ C_0>0$ or $\varepsilon^2 \pm 4 Q_1<0$, then from Lemma~\ref{Lem1} it follows that the equilibrium $(\rho_0^\pm,\phi_0^\pm)$ of the corresponding limiting system \eqref{limsys} is unstable.

If $\varepsilon\neq 0$, $|4 Q_1 |\leq \varepsilon^2$ and $(32 B_0+5\varepsilon^2) C_0<0$, then there are
\begin{gather*}
\rho_0^{k}=\sqrt{-\frac{32 B_0+5\varepsilon^2}{24  C_0}}, \quad \phi_0^{k}= (-1)^{k+1}\arcsin \frac{4 Q_1}{\varepsilon^2}+\pi k, \quad k\in\mathbb Z
\end{gather*}
such that assumption \eqref{as1} holds with
$\mathcal D_{n,q}^k=3  C_0 (  \varepsilon   \rho_0^k\cos 2\phi_0^k)^2/32\neq 0$. 
It can easily be checked that 
\begin{gather*}
\lambda_n^k=\frac{3 C_0(\rho_0^k)^2}{4}, \quad 
\xi_n^k=\eta_q^k=0, \quad 
\omega_q^k=\frac{  \varepsilon^2 \cos^2 2\phi_0^k}{8}.
\end{gather*}
In this case, $\beta_1^0=\lambda_n^k$ and $\tilde\beta_2^0=\beta_2^0=\omega_q^k$. We see that $\tilde \beta_2^0>0$. It follows from Lemma~\ref{Lem1} that for all $k\in\mathbb Z$ the equilibrium $(\rho_0^k,\phi_0^k)$ of the corresponding limiting system \eqref{limsys} is unstable.

Note that the parameter plane $( B_0, Q_1)$ is divided on the following parts (see~Fig.~\ref{ppEx0}):
$\mathfrak D_+:=\{(B_0,Q_1)\in\mathbb R^2: Q_1>g_+(B_0,\varepsilon)\}$,  
$\mathfrak D_-:=\{(B_0,Q_1)\in\mathbb R^2: Q_1<g_-(B_0,\varepsilon)\}$, 
$\Gamma_\pm:=\{(B_0,Q_1)\in\mathbb R^2: Q_1=g_\pm(B_0,\varepsilon)\}$, 
$P_0=(0,0)$, 
$P_{\varepsilon,0}=(-7\varepsilon^2/32,0)$, 
$P_{\varepsilon,\pm}=(-3\varepsilon^2/32,\pm\varepsilon^2/4)$ and 
$\mathfrak D_0=\mathbb R^2\setminus \overline{\mathfrak D_+ \cup \mathfrak D_-}$, where
\begin{gather*}
g_\pm( B_0,\varepsilon)\equiv 
	\begin{cases}
		\displaystyle \mp 2\left( B_0+\frac{7\varepsilon^2}{32}\right), & \displaystyle  B_0<-\frac{3\varepsilon^2}{32}, \\
		\displaystyle \mp\frac{\varepsilon^2}{4}, & \displaystyle B_0\geq -\frac{3\varepsilon^2}{32}.
	\end{cases}
\end{gather*}
Note that if $\varepsilon\neq 0$, then $\mathfrak D_+ \cap \mathfrak D_-$ is not empty. Thus, if $(B_0,Q_1)\in\mathfrak D_\pm$ and $C_0<0$, then a stable phase locking regime $\rho_L^\pm(t)$, $\phi_L^\pm(t)$ occurs in the corresponding truncated system \eqref{trsys} and persists in the full stochastic system \eqref{Ex0} at least on an asymptotically long time interval (see~Fig.~\ref{FigEx1}). 

\begin{figure}
\centering
\subfigure[$\varepsilon=0$]
    {
     \includegraphics[width=0.42\linewidth]{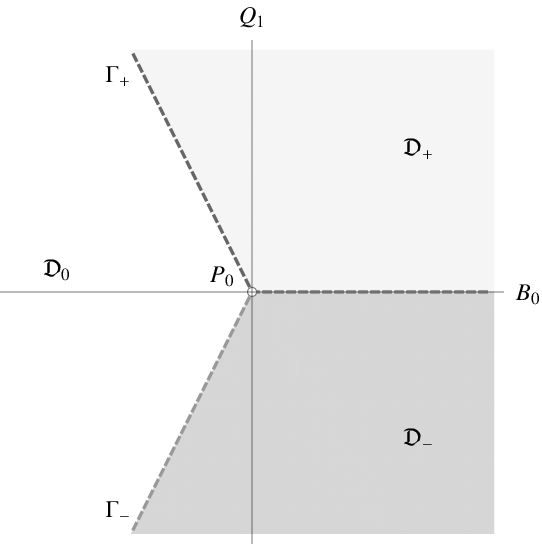}
    }
		\hspace{2ex}
\subfigure[$\varepsilon\neq 0$]
    {
     \includegraphics[width=0.42\linewidth]{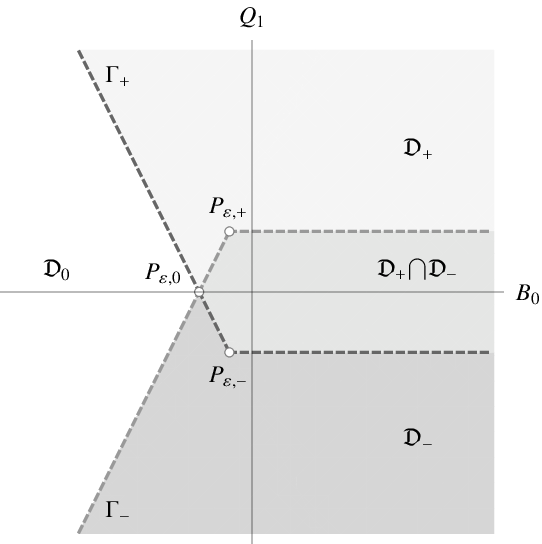}
    }
\caption{\footnotesize Partition of the parameter plane $(B_0,Q_1)$ for system \eqref{Ex0} with $n=2$, $p=1$, $s_0=2$ and $s_1=0$.} \label{ppEx0}
\end{figure}

\begin{figure}
\centering
{
   \includegraphics[width=0.4\linewidth]{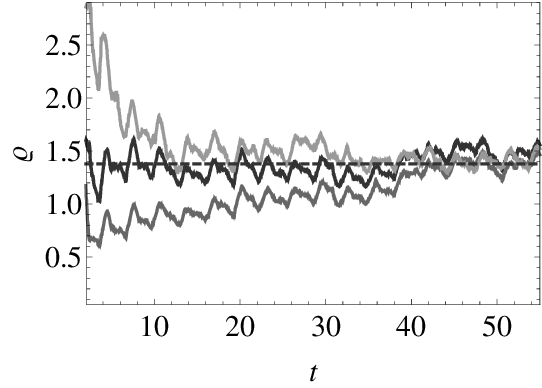}
}
\hspace{1ex}
{
   	\includegraphics[width=0.4\linewidth]{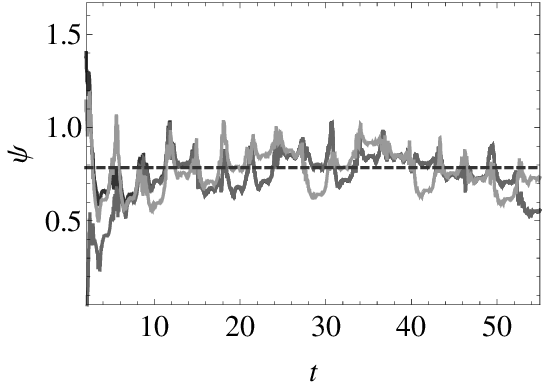}
}
\caption{\small The evolution of $\varrho(t)\equiv \sqrt{x_1^2(t)+x_2^2(t)}$ and $\psi(t)\equiv \varphi(t)-S(t)/2$, $\tan \varphi(t)=-x_2(t)/x_1(t)$ for solutions to system \eqref{Ex0} with $s_0=2$, $s_1=0$, $A_1=0$, $B_0=-1$, $B_1=2.5$, $C_0=-0.2$, $\varepsilon=0.4$ and different values of initial data. The dashed curves correspond to $\varrho(t)\equiv \rho_0^+$ and $\psi(t)\equiv -\pi/4$, where $\rho_0^+ \approx 1.38$.} \label{FigEx1}
\end{figure}

Now, let $s_1\neq 0$.  Then, assumption \eqref{asq} holds with $q=1$. It follows from Lemma~\ref{Lem4} that the system can be transformed into \eqref{rpsi2} with
\begin{gather*}
\mathcal F(z,\psi,t)= t^{-\frac{1}{2}}\mathcal F_2(z)+\widetilde{\mathcal F}_2(z,\psi,t), \quad 
\mathcal G(z,\psi,t)=- t^{-\frac{1}{4}}\frac{s_1}{2}+\widetilde{\mathcal G}_1(z,\psi,t),
\end{gather*}
where 
\begin{gather*}
\mathcal F_2(z)\equiv \frac{z}{32}\left(16 B_0+12 C_0 z^2+3\varepsilon^2\right),
\end{gather*}
and $\widetilde{\mathcal F}_2(z,\psi,t)=\mathcal O(t^{-3/4})$, $\widetilde{\mathcal G}_1(z,\psi,t)=\mathcal O(t^{-1/2})$ as $t\to\infty$ uniformly for all $z\in (0,\mathcal R]$, $\psi\in\mathbb R$ with any fixed $\mathcal R>0$. 

If $(16 B_0+3\varepsilon^2)  C_0<0$, then there exist
\begin{gather*}
\rho_0=\sqrt{-\frac{16 B_0+3\varepsilon^2}{12  C_0}}
\end{gather*}
such that assumption \eqref{as22} holds with
$\hat\xi_{n}=3  C_0 \rho_0^2/4\neq 0$. 
Therefore, if $ C_0<0$ and $16 B_0+3\varepsilon^2>0$, then it follows from Lemma~\ref{Lem5} that there is a phase drifting solution $\rho_D(t)$, $\phi_D(t)$ of the corresponding truncated system \eqref{trsys2} such that $\rho_D(t)\sim \rho_0$, $\phi_D(t)\sim -s_1 t^{1/4}$ as $t\to\infty$. Moreover, from Theorem~\ref{Th3} it follows that for all $l\in(0,1)$ this solution is partially stochastically stable in system \eqref{Ex0} on an asymptotically long time interval as  $t_s\leq t\leq  \varepsilon^{-4(1-l)}/4$ (see~Fig.~\ref{FigEx11}).

\begin{figure}
\centering
\subfigure[$\varepsilon=0$]
    {
     \includegraphics[width=0.42\linewidth]{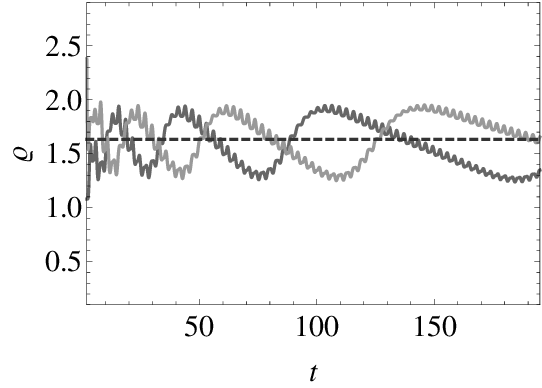}
    }
		\hspace{2ex}
\subfigure[$\varepsilon=0.5$]
    {
     \includegraphics[width=0.42\linewidth]{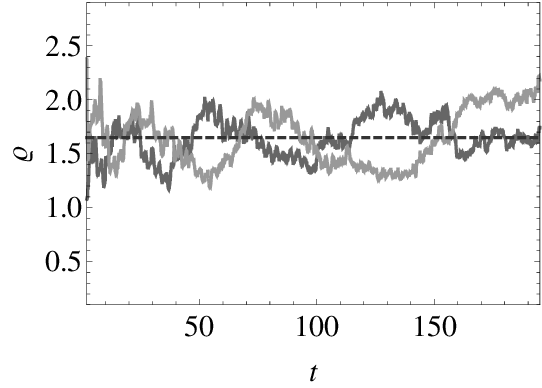}
    }
\caption{\small The evolution of $\varrho(t)\equiv \sqrt{x_1^2(t)+x_2^2(t)}$ for solutions to system \eqref{Ex0} with $s_0=2$, $s_1=8$, $A_1=0$, $B_0=1$, $B_1=1$, $C_0=-0.5$ and different values of $\varepsilon$ and initial data. The dashed curves correspond to $\varrho(t)\equiv \rho_0$, where $\rho_0=\sqrt{8/3+\varepsilon^2/2}$.} \label{FigEx11}
\end{figure}

\subsection{Example 2.}  
Consider a stochastically perturbed van der Pol-type equation with a damped parametric excitation (see~\cite[p. 283]{BM61}):
\begin{gather}\label{Ex2}
dx_1=x_2dt, \quad dx_2=\left(-x_1+t^{-\frac{n}{2}} f(x_1,x_2,S(t))\right) dt+\varepsilon t^{-\frac{p}{2}} g(x_1,S(t))\, dw_1(t),
\end{gather}
where
\begin{gather*}
f(x_1,x_2,S)\equiv A_1 x_1 \sin S +(B_0+B_1\sin 2S+C_0 x_1^2)x_2, \\ 
g(x,S)\equiv x \sin S, \quad S(t)\equiv t+ 2 s_1 t^{\frac {1}{2}}
\end{gather*}
with $A_1,B_0,B_1,C_0,s_1\in\mathbb R$. Note that system \eqref{Ex2} is not self-excited in the absence of parametric pumping $(A_1=B_1=\varepsilon=0)$ and $B_0<0$ (see~Fig.~\ref{FigEx20}, a). Let us show that decaying oscillatory perturbations can lead to stable resonant solutions with an asymptotically constant amplitude.

\begin{figure}
\centering
\subfigure[$B_1=0$]
{\includegraphics[width=0.4\linewidth]{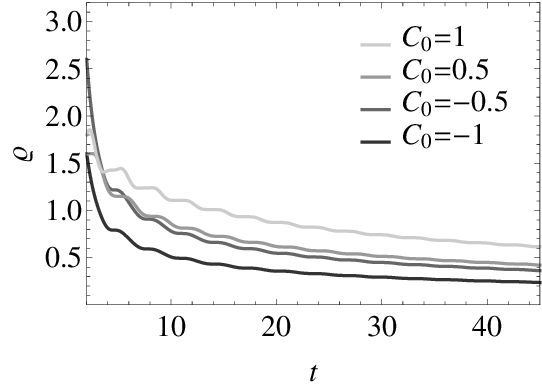} }
		\hspace{2ex}
		\subfigure[$C_0=-1$]
{\includegraphics[width=0.4\linewidth]{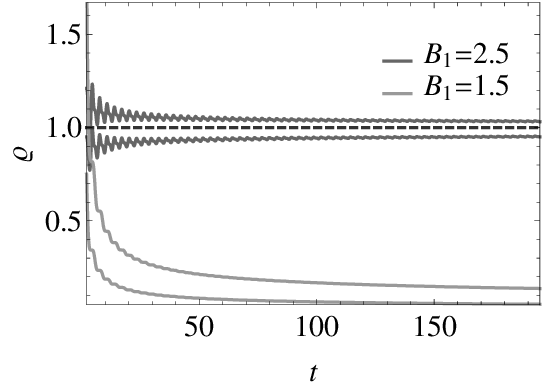} }
\caption{\small 
The evolution of $\varrho(t)\equiv \sqrt{x_1^2(t)+x_2^2(t)}$ for sample paths of solutions to system \eqref{Ex2} with $n=2$, $p=1$, $A_1=0$, $B_0=-1$, $\varepsilon=0$, and different values of the parameters $B_1$, $C_0$, and initial data.} \label{FigEx20}
\end{figure}

It can easily be checked that system \eqref{Ex2} corresponds to \eqref{PS} with $s_0=1$, $\nu_0=1$, $\mu(t)\equiv t^{-1/2}$, and ${\bf a}(\varrho,\varphi,S,t)=(a_1(\varrho,\varphi,S,t),a_2(\varrho,\varphi,S,t))^T$, ${\bf A}(\varrho,\varphi,S,t)=\{\alpha_{i,j}(\varrho,\varphi,S,t)\}_{2\times 2}$ defined by \eqref{fgex0str} and \eqref{fgex0}. It can easily be checked that the conditions \eqref{rc} and \eqref{mucond} holds with $\kappa=1$, $\varkappa=2$, $m=2$ and $\chi_m=-1/2$. 

Let $n=2$ and $p=1$. Then, the transformations described in Theorem~\ref{Th1} with $N=2$ reduce the system to \eqref{rpsi} with 
\begin{align*}
\Lambda(r,\psi,S,t) \equiv &\,  t^{-1}\Lambda_2(r,\psi)+\widetilde \Lambda_1(r,\psi,S,t), \\ 
\Omega(r,\psi,S,t)\equiv &\, t^{-\frac{1}{2}}\Omega_1(r,\psi) + t^{-1}\Omega_2(r,\psi)+\widetilde \Omega_2(r,\psi,S,t),
\end{align*}
where
\begin{align*}
\Lambda_2(r,\psi)\equiv & \, \frac{r}{32} \left(16 B_0 + 4  C_0 r^2 + 3\varepsilon^2 -2 Q_1 \cos (2\psi+\vartheta_0) \right), \\ 
\Omega_1(r,\psi)\equiv & \, - s_1, \\
\Omega_2(r,\psi)\equiv & \, \frac{Q_1}{16} \sin( 2 \psi +\vartheta_0),
\end{align*}
and $\widetilde \Lambda_2(r,\psi,S,t)=\mathcal O(t^{-3/2})$, $\widetilde \Omega_2(r,\psi,S,t)=\mathcal O(t^{-3/2})$ as $t\to\infty$ uniformly for all $r\in (0,\mathcal R]$, $(\psi,S)\in\mathbb R^2$ with any fixed $\mathcal R>0$, 
\begin{gather*}
Q_1=\sqrt{16 B_1^2+\varepsilon^4}, \quad \vartheta_0=\arcsin \frac{4B_1}{D_1}.
\end{gather*}

Let $s_1=0$. Then, assumption \eqref{asq} holds with $q=n=2$.  
If $(16 B_0+3\varepsilon^2+2Q_1) C_0<0$, then there exist
\begin{gather*}
\rho_0=\sqrt{-\frac{16 B_0+2Q_1+3\varepsilon^2}{4  C_0}}, \quad \phi_0 = \frac{\pi}{2}-\frac{\vartheta_0}{2}+\pi k, \quad k\in\mathbb R
\end{gather*}
such that assumption \eqref{as1} holds with $\mathcal D_{n,q}^\pm=- C_0 Q_1\rho^2 /32\neq 0$. It can easily be checked that 
\begin{gather*}
\lambda_n^\pm=\frac{3 C_0 \rho_0^2}{4}, \quad 
\xi_n^\pm=\eta_q^\pm=0, \quad 
\omega_q^\pm=-\frac{Q_1}{8}.
\end{gather*}
In this case, $\beta_1^0=\lambda_n^\pm$ and $\tilde\beta_2^0=\beta_2^0=\omega_q^\pm$. If, in addition, $ C_0<0$ and $B_0>-(2Q_1+3\varepsilon^2)/16$, then it follows from Lemma~\ref{Lem2} that there is a stable phase locking solution $\rho_L(t)$, $\phi_L(t)$ of the corresponding truncated system \eqref{trsys} such that $\rho_L(t)\sim \rho_0$, $\phi_L(t)\sim \phi_0$ as $t\to\infty$. It follows from Theorem~\ref{Th2} that for all $l\in(0,1)$ this regime is stochastically stable in system \eqref{Ex2} on an asymptotically long time interval as  $t_s\leq t\leq t_s \exp \varepsilon^{-2(1-l)}$. Note that if $ C_0>0$ and $B_0<-(2Q_1+3\varepsilon^2)/16$, then from Lemma~\ref{Lem1} it follows that the equilibrium $(\rho_0,\phi_0)$ of the corresponding limiting system \eqref{limsys} is unstable. 

In this case, the parameter plane $( B_0, B_1)$ is divided on the following parts (see~Fig.~\ref{ppEx2}):
$\mathfrak D:=\{(B_0,B_1)\in\mathbb R^2: B_0>g(B_1,\varepsilon)\}$, 
$\Gamma:=\{(B_0,B_1)\in\mathbb R^2: B_0=g(B_1,\varepsilon)\}$, 
$P_0=(0,0)$, 
$P_{\varepsilon}=(-5\varepsilon^2/16,0)$ and 
$\mathfrak D_0=\mathbb R^2\setminus \overline{\mathfrak D}$, where
\begin{gather*}
g( B_1,\varepsilon)\equiv 
	- \frac{3\varepsilon^2+2\sqrt{16 B_1^2+\varepsilon^4}}{16}.
\end{gather*}
Thus, if $(B_0,B_1)\in\mathfrak D_\pm$ and $C_0<0$, then a stable phase locking regime $\rho_L(t)$, $\phi_L(t)$ occurs in the corresponding truncated system \eqref{trsys} and persists in the full stochastic system \eqref{Ex2} at least on an asymptotically long time interval (see~Fig.~\ref{FigEx20}, b and Fig.~\ref{FigEx2}).

\begin{figure}
\centering
\subfigure[$\varepsilon=0$]
    {
     \includegraphics[width=0.42\linewidth]{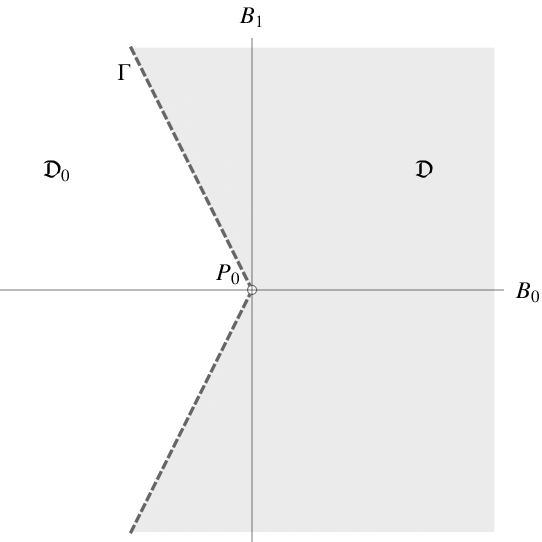}
    }
		\hspace{2ex}
\subfigure[$\varepsilon\neq 0$]
    {
     \includegraphics[width=0.42\linewidth]{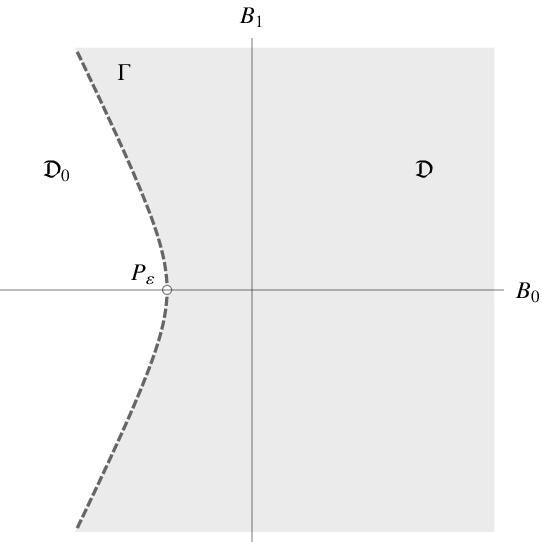}
    }
\caption{\footnotesize Partition of the parameter plane $(B_0,B_1)$ for system \eqref{Ex2} with $n=2$, $p=1$ and $s_1=0$.} \label{ppEx2}
\end{figure}

\begin{figure}
\centering
{
   \includegraphics[width=0.4\linewidth]{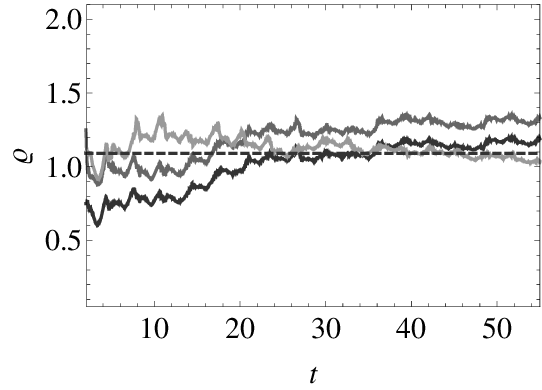}
}
\hspace{1ex}
{
  	\includegraphics[width=0.4\linewidth]{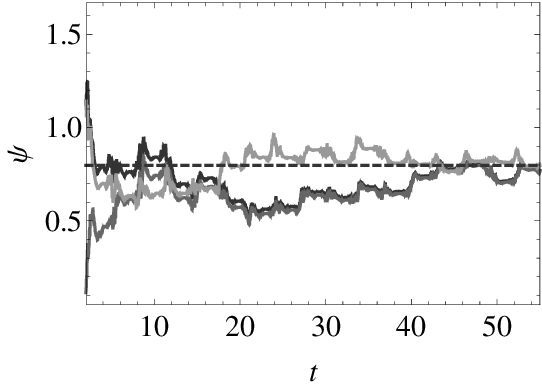}
}
\caption{\small The evolution of $\varrho(t)\equiv \sqrt{x_1^2(t)+x_2^2(t)}$ and $\psi(t)\equiv \varphi(t)-S(t)$, $\tan \varphi(t)=-x_2(t)/x_1(t)$ for solutions to system \eqref{Ex2} with $s_1=0$, $A_1=0$, $B_0=-1$, $B_1=2.5$, $C_0=-1$, $\varepsilon=0.5$ and different values of initial data. The dashed curves correspond to $\varrho(t)\equiv \rho_0$ and $\psi(t)\equiv \pi/2-\vartheta_0/2$, where $\rho_0 \approx 1.09$, $\pi/2-\vartheta_0/2\approx 0.79$.} \label{FigEx2}
\end{figure}

Now, let $s_1\neq 0$.  Then, assumption \eqref{asq} holds with $q=1$. It follows from Lemma~\ref{Lem4} that the system can be transformed into \eqref{rpsi2} with
\begin{gather*}
\mathcal F(z,\psi,t)= t^{-1}\mathcal F_2(z)+\widetilde{\mathcal F}_2(z,\psi,t), \quad 
\mathcal G(z,\psi,t)=- t^{-\frac{1}{2}}s_1+\widetilde{\mathcal G}_2(z,\psi,t),
\end{gather*}
where 
\begin{gather*}
\mathcal F_2(z)\equiv \frac{z}{32}\left(16 B_0+4 C_0 z^2+3\varepsilon^2\right),
\end{gather*}
and $\widetilde{\mathcal F}_2(z,\psi,t)=\mathcal O(t^{-3/2})$, $\widetilde{\mathcal G}_1(z,\psi,t)=\mathcal O(t^{-1/2})$ as $t\to\infty$ uniformly for all $z\in (0,\mathcal R]$, $\psi\in\mathbb R$ with any fixed $\mathcal R>0$. If $(16 B_0+3\varepsilon^2)  C_0<0$, then there exist
\begin{gather*}
\rho_0=\sqrt{-\frac{16 B_0+3\varepsilon^2}{4 C_0}}
\end{gather*}
such that assumption \eqref{as22} holds with
$\hat\xi_{n}=3  C_0 \rho_0^2/4\neq 0$. 
Therefore, if $ C_0<0$ and $16 B_0+3\varepsilon^2>0$, then it follows from Lemma~\ref{Lem5} that there is a phase drifting solution $\rho_D(t)$, $\phi_D(t)$ of the corresponding truncated system \eqref{trsys2} such that $\rho_D(t)\sim \rho_0$, $\phi_D(t)\sim -2 s_1 t^{1/2}$ as $t\to\infty$. From Theorem~\ref{Th3} it follows that for all $l\in(0,1)$ this solution is partially stochastically stable in system \eqref{Ex2} on an asymptotically long time interval as  $t_s\leq t\leq t_s \exp \varepsilon^{-2(1-l)}$ (see~Fig.~\ref{FigEx22}).

\begin{figure}
\centering
\subfigure[$\varepsilon=0$]
    {
     \includegraphics[width=0.42\linewidth]{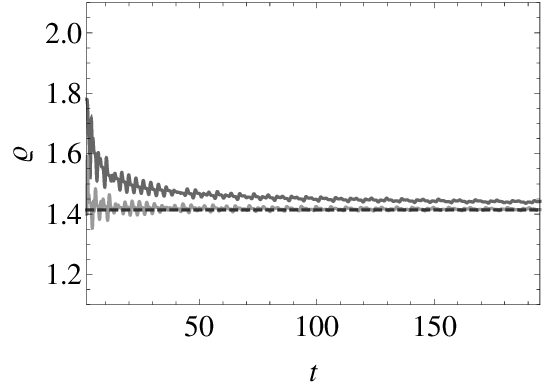}
    }
		\hspace{2ex}
\subfigure[$\varepsilon=0.5$]
    {
     \includegraphics[width=0.42\linewidth]{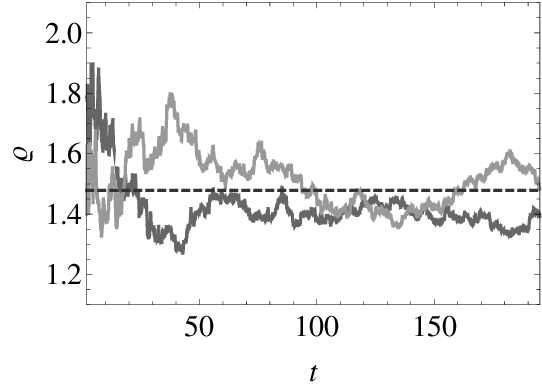}
    }
\caption{\small The evolution of $\varrho(t)\equiv \sqrt{x_1^2(t)+x_2^2(t)}$ for solutions to system \eqref{Ex2} with  $s_1=5$, $A_1=0$, $B_0=0.5$, $B_1=1$, $C_0=-1$ and different values of $\varepsilon$ and initial data. The dashed curves correspond to $\varrho(t)\equiv \rho_0$, where $\rho_0=\sqrt{2+3\varepsilon^2/4}$.} \label{FigEx22}
\end{figure}

}

\begin{thebibliography}{99}

\bibitem{FC08}  F. Calogero, \textit{Isochronous Systems}, Oxford Press, Oxford, 2008.

\bibitem{BM61} 
N. N. Bogolubov, Yu. A. Mitropolsky, \textit{Asymptotic Methods in Theory of Non-linear Oscillations}, Gordon and Breach, New York, 1961.

\bibitem{SJJ92} 
J.~J. Stoker, \textit{Nonlinear Vibrations in Mechanical and Electrical Systems}, Wiley, New York, 1992.

\bibitem{PRK01} 
A. Pikovsky, M. Rosenblum, J. Kurths, \textit{Synchronization: a Universal Concept in Nonlinear Sciences}, Cambridge University Press, Cambridge, 2001.

\bibitem{SS18} 
S.~H. Strogatz, \textit{Nonlinear Dynamics and Chaos: With Applications to Physics, Biology, Chemistry, and Engineering}, CRC Press, London, 2018.

\bibitem{Liu09} B. Liu,  \textit{Quasi-periodic solutions of forced isochronous oscillators at resonance}, J. Differential Equations, 246 (2009), 3471--3495.

\bibitem{BF09} D. Bonheure, C. Fabry, \textit{Littlewood’s problem for isochronous oscillators}, Arch. Math. 93 (2009), 379--388.

\bibitem{OR19} R. Ortega, D. Rojas, \textit{Periodic oscillators, isochronous centers and resonance}, Nonlinearity, 32 (2019), 800--832.

\bibitem{FF05unb} C. Fabry, A. Fonda, \textit{Unbounded motions of perturbed isochronous Hamiltonian systems at resonance}, Advanced Nonlinear Studies, 5 (2005), 351--373.

\bibitem{R20} D. Rojas, \textit{Resonance of bounded isochronous oscillators}, Nonlinear Analysis, 192 (2020), 111680.

\bibitem{LRS02} J. A. Langa, J. C. Robinson, A. Su\'{a}rez, \textit{Stability, instability and bifurcation phenomena in nonautonomous differential equations}, Nonlinearity, 15 (2002), 887--903.

\bibitem{KS05} P. E. Kloeden, S. Siegmund, \textit{Bifurcations and continuous transitions of attractors in autonomous and nonautonomous systems}, Internat. J. Bifur. Chaos., 15 (2005), 743--762.

\bibitem{MR08} M. Rasmussen, \textit{Bifurcations of asymptotically autonomous differential equations}, 
Set-Valued Anal., 16 (2008), 821--849.

\bibitem{LM56} L. Markus, Asymptotically autonomous differential systems, \emph{Contributions to the Theory of Nonlinear Oscillations III}, Ann. Math. Stud., 36, Princeton University Press, Princeton, 1956, 17--29.

\bibitem{LDP74} L. D. Pustyl'nikov, \textit{Stable and oscillating motions in nonautonomous dynamical systems. A generalization of C. L. Siegel's theorem to the nonautonomous case}, Math. USSR-Sbornik, 23 (1974), 382--404.

\bibitem{HRT94} H. Thieme, \textit{Asymptotically autonomous differential equations in the plane}, Rocky Mountain J. Math., 24 (1994), 351--380.


\bibitem{BN10} V. Burd, P. Nesterov, \textit{Parametric resonance in adiabatic oscillators}, Results. Math., 58 (2010), 1--15.

\bibitem{PN13} P. Nesterov,  \textit{Appearance of new parametric resonances in time-dependent harmonic oscillator}, Results. Math. 64 (2013), 229--251.

\bibitem{OS21DCDS} O. A. Sultanov, \textit{Bifurcations in asymptotically autonomous Hamiltonian systems under oscillatory perturbations}, Discrete and Continuous Dynamical Systems, 41 (2021), 5943--5978.

\bibitem{OS24QTDS} O. A. Sultanov, \textit{Resonance in isochronous systems with decaying oscillatory perturbations}, Qualitative Theory of Dynamical Systems, 23 (2024), 295.

\bibitem{OS25DCDS} O. A. Sultanov, \textit{Nonlinear resonance in oscillatory systems with decaying perturbations},  Discrete and Continuous Dynamical Systems, 45 (2025), 1691--1719.

\bibitem{FW98} M. I.~Freidlin, A. D.~Wentzell, \textit{Random Perturbations of Dynamical Systems}, Springer-Verlag, New York, Heidelberg, Berlin, 1998.

\bibitem{MD96} 
M. Dimentberg,  \textit{Random vibrations of an isochronous SDOF bilinear system}, Nonlinear Dyn., 11 (1996), 401--405.

\bibitem{BST94}
A. Bazzani, S. Siboni, G. Turchetti, \textit{Diffusion in Hamiltonian systems with a small stochastic perturbation}, Physica D, 76 (1994), 8--21.

\bibitem{BL22} E. Bernardi, A. Lanconelli, \textit{Stochastic perturbation of a cubic anharmonic oscillator}, Discrete and Continuous Dynamical Systems - B, 27 (2022), 2563--2585.

\bibitem{MF03} M. Freidlin, \textit{Autonomous stochastic perturbations of dynamical systems}, Acta Applicandae Mathematicae, 78 (2003), 121--128.

\bibitem{AIN04} L. Arnold, P. Imkeller, N.Sri Namachchivaya, \textit{The asymptotic stability of a noisy non-linear oscillator}, Journal of Sound and Vibration, 269 (2004), 1003--1029.

\bibitem{PHB04} P. H. Baxendale, \textit{Stochastic averaging and asymptotic behavior of the stochastic Duffing–van der Pol equation}, Stochastic Processes and their Applications, 113 (2004), 235--272.

\bibitem{RNI86} R.N. Iyengar, \textit{A nonlinear system under combined periodic and random excitation}, J Stat. Phys., 44 (1986), 907--920.

\bibitem{HZS00} Z.L. Huang, W.Q. Zhu, Y. Suzuki, \textit{Stochastic averaging of strongly non-linear oscillators under combined harmonic and white-noise excitations}, J. Sound and Vibration, 238 (2000), 233--256.

\bibitem{ZW03}
W. Q. Zhu, Y. J. Wu, \textit{First-passage time of Duffing oscillator under combined harmonic and white-noise excitations}, Nonlinear Dyn. 32 (2003), 291--305.

\bibitem{HZ04}
Z.L. Huang, W.Q. Zhu, \textit{Stochastic averaging of quasi-integrable Hamiltonian systems under combined harmonic and white noise excitations,}
International Journal of Non-Linear Mechanics, 39 (2004), 1421--1434

\bibitem{CNO11}
S. Choi, N. Sri Namachchivaya, K. Onu, \textit{Resonant dynamics of a periodically driven noisy oscillator}, Probabilistic Engineering Mechanics, 26 (2011), 109–118.

\bibitem{AGR09} J. A. D.~Appleby, J. P.~Gleeson, A.~Rodkina, \textit{On asymptotic stability and instability with respect to a fading stochastic perturbation}, Appl. Anal., 88 (2009), 579--603.

\bibitem{ACR11} J. A. D.~Appleby, J.~Cheng, A.~Rodkina, \textit{Characterisation of the asymptotic behaviour of scalar linear differential equations with respect to a fading stochastic perturbation}, Discrete Contin. Dyn. Syst., Supplement (2011), 79--90.

\bibitem{KT13} O. I.~Klesov,  O. A.~Tymoshenko, \textit{Unbounded solutions of stochastic differential equations with time-dependent coefficients}, Annales Univ. Sci. Budapest., Sect. Comp., 41 (2013), 25--35.

\bibitem{OS22IJBC} O. A. Sultanov, \textit{Bifurcations in asymptotically autonomous Hamiltonian systems subject to multiplicative noise}, International Journal of Bifurcation and Chaos, 32 (2022), 2250164.

\bibitem{OS24CPAA} O. A. Sultanov, \textit{Stability of asymptotically Hamiltonian systems with damped oscillatory and stochastic perturbations}, Communications on Pure and Applied Analysis, 23 (2024), 432--462.

\bibitem{OS23SIAM} O. A.~Sultanov, \textit{Long-term behaviour of asymptotically autonomous Hamiltonian systems with multiplicative noise}, SIAM Journal on Applied Dynamical Systems, 22 (2023), 1818--1851.

\bibitem{OS25CNSNS} O. A. Sultanov, \textit{Resonances in nonlinear systems with a decaying chirped-frequency excitation and noise}, Communications in Nonlinear Science and Numerical Simulation, 145 (2025), 108713.

\bibitem{BVC79} 
B. V. Chirikov, \textit{A universal instability of many-dimensional oscillator systems}, Physics Reports, 52 (1979), 263--379.
 
\bibitem{KF13} V. V. Kozlov, S. D. Furta, \textit{ Asymptotic Solutions of Strongly Nonlinear Systems of Differential Equations}, Springer Monographs in Mathematics, Springer, New York, 2013.

\bibitem{RH12} R.~Khasminskii, \textit{Stochastic Stability of Differential Equations}, Springer, Berlin, Heidelberg, 2012.

\bibitem{HJK67} H. J.~Kushner,  \textit{Stochastic Stability and Control}, Academic Press, New York, 1967.

\bibitem{VIV98} V. I.~Vorotnikov, \textit{ Partial Stability and Control}, Birkh\"{a}user, Boston, Basel, Berlin, 1998.

\bibitem{BO98} {\O}ksendal, B. \textit{Stochastic Differential Equations. An Introduction with Applications}, Springer, New York, Heidelberg, Berlin, 1998.

\bibitem{Hap93} M. M. Hapaev, \textit{Averaging in stability theory: a study of resonance multi-frequency systems}, Kluwer Academic Publishers, Dordrecht, Boston, 1993.

\bibitem{HKH} H. K. Khalil, \textit{Nonlinear Systems}, Prentice Hall, Upper Saddle River, New Jersey, 2002.

\end{thebibliography}
\end{document}